\documentclass[12pt]{amsart}
        \usepackage{amsmath}
\usepackage{amsfonts}
\usepackage{amsthm}
\usepackage{amssymb}
\usepackage{amscd}
\usepackage[all]{xy}
\usepackage{enumerate}
\usepackage{hyperref}
\usepackage{comment}
\usepackage{paralist}
\usepackage{xcolor}


\makeatletter
\@namedef{subjclassname@2020}{\textup{2020} Mathematics Subject Classification}
\makeatother


\theoremstyle{plain}
\newtheorem{thm}{Theorem}[section]

\newtheorem{prop}[thm]{Proposition}

\newtheorem{cor}[thm]{Corollary}
\newtheorem{lem}[thm]{Lemma}

\theoremstyle{definition}

\newtheorem{rem}[thm]{Remark}

\newtheorem{expl}[thm]{Example}

\newtheorem{setup}[thm]{Setup}

\theoremstyle{remark}
\newtheorem{claim}{Claim}[thm]

\newcommand{\bP}{{\bf P}}
\newcommand{\bF}{{\bf F}}
\newcommand{\fd}{\mathfrak{d}}
\newcommand{\fe}{\mathfrak{e}}
\newcommand{\sA}{\mathcal{A}}

\newcommand{\sE}{\mathcal{E}}

\newcommand{\sG}{\mathcal{G}}
\newcommand{\sH}{\mathcal{H}}

\newcommand{\sL}{\mathcal{L}}

\newcommand{\sM}{\mathcal{M}}
\newcommand{\sO}{\mathcal{O}}

\newcommand{\sQ}{\mathcal{Q}}

\newcommand{\alb}{\mathrm{alb}}

\newcommand{\Sym}{\mathrm{Sym}}

\usepackage{color}

\def\geq{\geqslant}
\def\leq{\leqslant}

\renewcommand{\P}{\mathbf{P}}
\newcommand{\N}{\mathbf{N}}
\renewcommand{\O}{\mathcal{O}}
\let\subset\subseteq
\let\epsilon\varepsilon

\makeatletter
\newcommand{\restr}[2]{\left. #1 \right| _{#2}}
\def\@restrpar[#1]#2{\left. (#1) \right| _{#2}}
\def\@restrst#1#2{\left. #1 \right| _{#2}}
\def\restr{\@ifnextchar[{\@restrpar}{\@restrst}}
\makeatother

\numberwithin{equation}{section}

\newcommand{\beba}  {\begin{equation}\begin{array}{rcl}}

\newcommand{\eaee}  {\end{array}\end{equation}}

\makeatletter
\def\l@section{\@tocline{1}{0pt}{1pc}{}{}}
\def\l@subsection{\@tocline{2}{0pt}{1pc}{4.6em}{}}
\def\l@subsection{\@tocline{3}{0pt}{1pc}{7.6em}{}}
\renewcommand{\tocsection}[3]{%
  \indentlabel{\@ifnotempty{#2}{\makebox[2.3em][l]{%
    \ignorespaces#1 #2.\hfill}}}#3}
\renewcommand{\tocsubsection}[3]{%
  \indentlabel{\@ifnotempty{#2}{\hspace*{2.3em}\makebox[2.3em][l]{%
    \ignorespaces#1 #2.\hfill}}}#3}
\renewcommand{\tocsubsection}[3]{%
  \indentlabel{\@ifnotempty{#2}{\hspace*{4.6em}\makebox[3em][l]{%
    \ignorespaces#1 #2.\hfill}}}#3}
\makeatother

\title{On surfaces of
  high degree with respect to the sectional genus}

\author{Ciro Ciliberto, Thomas Dedieu, Margarida Mendes Lopes}
\address{
Ciro Ciliberto\\Dipartimento di Matematica \\
Universit\`a di Roma ``Tor Vergata''\\
Via della Ricerca  Scientifica, 00177 Roma\\ Italia
\texttt{cilibert@axp.mat.uniroma2.it   }}

\address{Thomas Dedieu\\
Institut de Math{\'e}matiques de Toulouse; UMR5219\\
Universit{\'e} de Toulouse; CNRS\\
UPS IMT, F-31062 Toulouse Cedex 9\\ France
\texttt{thomas.dedieu@math.univ-toulouse.fr}}

\address{Margarida Mendes Lopes\\
Centro de An\'alise Matem\'atica, Geometria e Sistemas Din\^amicos \\
Departamento de Matem\'atica\\
Instituto Superior T\'ecnico \\
Universidade de Lisboa \\
Av. Rovisco Pais 1, 1049--001 Lisboa\\ Portugal
\texttt{mmendeslopes@tecnico.ulisboa.pt}}

\begin{document}

\begin{abstract}
We study and classify linearly normal surfaces in $\P^n$, of degree
$d$ and sectional genus $g$, such that $d\geq 2g-1$.
\end{abstract}

\maketitle

\tableofcontents{}

\section{Introduction}

Let $S$ be a complex, projective, linearly normal, irreducible and
non-degenerate surface in $\P^n$, arbitrarily singular,
of degree $d$ and sectional (geometric) genus $g$. This paper is
devoted to the study and classification of such surfaces under the
hypothesis that $d\geq 2g-1$, which implies that the Kodaira dimension
of $S$ is $-\infty$.

This problem is very natural, and interesting for its own sake.
Yet our interest in it grew out of the study of extensions of
curves: let $C$ be a smooth projective curve in $\P^{n-1}$; we want to
undertand all the surfaces in $\P^n$ (if any) that have $C$ as a
hyperplane section. In this perspective it is important to consider
surfaces that may be singular, which was forbidden by assumption in
most of the previous results on the classification problem considered
in this 
article. Moreover, the assumption that $d\geq 2g-1$ is fairly natural
in this context. Besides note that starting from degree $2g-2$, new
kind of objects show up, e.g., canonical curves and $K3$ surfaces.
Finally, it turns out to be convenient in this context to assume that
$C$ is linearly normal (however we don't make this assumption in
general in the present article). 

One of the main ideas
in this article
is very classical, and consists in using the adjoint and pluriadjoint
systems $|mK_{S'}+H|$ of the pull back $|H|$ of the  hyperplane system
on a minimal desingularization $S'$ of $S$.
To start with, we recall
(and elaborate on) the classical fact that $|K_{S'}+H|$ is empty if
and only if $S$ is either a scroll or the Veronese surface of degree 4
in $\P^5$, see Section~\ref{S:vanish-adj1}.
Then we assume for the rest of the article that
$|K_{S'}+H|$ is non--empty.

Under this assumption we prove that
$K_{S'}+H$ is nef, hence $(K_{S'}+H)^2\geq 0$; then we have separate
discussions depending on whether the latter inequality is strict or
not.
On the one hand, we classify the
cases in which $(K_{S'}+H)^2=0$:  
then, either $g=1$ and $S$ is a weak Del Pezzo surface, or
$g\geq 2$ and $S$ is ruled by conics (see Section \ref {s:4.2}).
On the other hand, one of
the main results of the article is a sufficient condition (always
verified if the surface is regular) for  $|K_{S'}+H|$ to be base point
free
if $(K_{S'}+H)^2>0$ (see Proposition \ref {connected}). 
In Section \ref  {s:examples} we give examples showing the sharpness
of this result. We further prove that if
$(K_{S'}+H)^2> 0$ and $d\geq 10$, then $|K_{S'}+H|$ determines a birational map (see Proposition \ref {lbirational}). 

In Section \ref {S:irreg}  we focus on the case when the surface is
irregular (i.e., $q>0$).  The main result
in this case is that
if $(K_{S'}+H)^2>0$, then $g-q\geq 3$, which we show
implies that $|K_{S'}+H|$ is not composed with a pencil (see Theorem
\ref {thm:npp}).  In Section \ref {S:supraCSegre} we go on studying
the irrational case and prove an extension of a classical result by
C. Segre: the latter says that if a scroll has linearly normal hyperplane  
sections of positive genus, then it is a cone 
(see \cite[Thm. 2.3]{CD}), while our result says that
if $S$ is ruled by conics and has linearly normal hyperplane  
sections, then it is essentially the 2--Veronese re-embedding of a
cone if  $d>2g+5$ (see Theorem \ref {thm:segre}). Moreover we show that if
$d>3g-3$,
then $S$ is always ruled by lines or conics, hence it is either a cone or  a 2--Veronese thereof if it has linearly normal hyperplane  
sections (see Lemma \ref {prop:irr} and Corollary \ref {cor:irr}). 
 
Finally we focus on rational surfaces.
In this case
we give a complete classification of surfaces for which the biadjoint
system $|2K_{S'}+H|$ (see Theorem \ref {thm:class1}), or the triadjoint
system $|3K_{S'}+H|$ is empty (see Theorems \ref {thm:class2} and \ref
{thm:class3});
recall that the case when the adjoint system is empty has been sorted
out from the beginning.

The latter results on rational surfaces, plus the aforementioned extension
of Segre's theorem, give a complete classification of surfaces of degree
$d>3g-3$ with linearly normal hyperplane sections,
which we reckon is the main output of this article. The classification
for $d\geq 4g-4$ had been given previously by the two first-named
authors in \cite{CD}.
Synthetically, the classification gives
the following list:
Veronese surfaces represented by plane curves of
degree at most 8,
Del Pezzo surfaces or 2--Veronese images thereof,
and surfaces represented by linear systems of $k$--gonal
curves, with $k\leq 5$. 

One of our important tools, besides projective and birational techniques,  is an analogue, or rather a slight strengthening, of Reider's method, which is based on a detailed study of properties of $m$--connected effective divisors on a smooth surface (see in particular Lemma \ref {l:1-conn}, due to the third-named author). 

Finally we point out that, on the way, we provide many auxiliary
results, too many to be described here,  that are not strictly
necessary for us, but
will be useful in future work, we believe.

\medskip
{\bf Acknowledgements}: The first author is a member of GNSAGA of the
Istituto Nazionale di Alta Matematica ``F. Severi''.
The second author acknowledges
support from the ANR project FanoHK, grant ANR-20-CE40-0023.
The third author is a member of Centro de An\'alise Matem\'atica, Geometria e Sistemas Din\^amicos and this work was partially funded by FCT/Portugal through project UIDB/04459/2020 with DOI identifier 10-54499/UIDP/04459/2020.

The authors thank the anonymous referee for the careful reading of the paper and for the appropriate remarks and suggestions that helped correct some minor mistakes and improve the exposition in a few places.

\section{Setup and notation} \label{sec:prel}

\subsection{Setup and notation} \label{ssec:setup} In this paper we
 consider irreducible complex projective surfaces $S \subset \P^n$ that are  linearly
normal and non-degenerate; we let $d$ be the degree of
$S$, and assume $n\geq 3$.
We set $d=n+a$, with $a \geq -1$.
Let  $\pi: S'\longrightarrow  S$
be  the minimal desingularization of $S$ and  set
$H:=\pi^*(\sO_{S}(1))$; it is a nef and big line bundle on $S'$,
generated by its global sections.
The general curve $C\in |H|$ is smooth and irreducible, and we 
denote  its genus by $g$, that will be called the \emph{sectional genus}  of $S$. Note that $H^2=d$.
By the minimality assumption, there is no $(-1)$-curve $\theta$ on $S'$  such that $H\cdot \theta=0$. 

We assume $d\geq 2g-1$.  Since $2g-2=K_{S'}\cdot H+ H^2$, this
assumption is equivalent to the assumption $K_{S'}\cdot H<0$.  Since
$H$ is nef, $K_{S'}\cdot H<0$ implies that $h^0(S', mK_{S'})=0$ for
all $m\in \N$, and so the Kodaira dimension of $S'$ is
$-\infty$.

We  set $q:=q(S')$ the \emph{irregularity} of $S'$.
If $q>0$, since $\kappa(S') = -\infty$,
the \emph{Albanese morphism} of $S'$ is a morphism
$\alb: S'\longrightarrow \Gamma$, where $\Gamma$ is a smooth,
projective curve of genus $q$, and the general fibre $G$ of $\alb$
is a smooth rational curve. The fibres of $\alb$ form a pencil of
genus $q$, called the \emph{Albanese pencil} of $S'$.

For all positive integers $m$, we call the  linear system
$|mK_{S'}+H|$ 
 the \emph{$m$--adjoint system} to $H$.

\subsection{Further notation} For all $e\in {\bf N}$, we let $\bF_e$ be the rational ruled surface
$\bP(\sO_{\bP^1}\oplus \sO_{\bP^1}(-e))$, and we denote by
$E$  a section with self-intersection $-e$
(in case $e>0$ this  section is  unique),   and
by $F$ the class of the fibres of the structure map $\bF_e\longrightarrow \bP^1$.

For all $d\in {\bf N}$, we will denote by $v_d$ the \emph{$d$--Veronese map} and by $V_d\subset \bP^{\frac {d(d+3)}2}$ the \emph{$d$--Veronese image of $\bP^2$}. 

A complete linear system of plane curves of degree $d$ with $n$ base points of multiplicities $m_1,\ldots, m_n$ will be denoted by $(d;m_1,\ldots, m_n)$. We will use the exponential notation $(d;m^{i_1}_1,\ldots, m^{i_n}_n)$ for repeated multiplicities. 

\def\lineq{\sim}
\def\numeq{\equiv}
We will denote by $\lineq$ the linear equivalence and by $\numeq$ the
numerical equivalence of divisors.

For all $m\in \N$, an effective divisor $D$  on a smooth,
irreducible, projective surface $T$ is
\emph{$m$-connected} if, for every decomposition
$D=D_1+D_2$ with $D_1$ and $D_2$ effective and non-zero, one has
$D_1\cdot D_2 \geq m$.

We call \emph{$(-1)$-divisor} a $1$-connected divisor $D$ on a smooth,
irreducible, projective surface
$T$ such that $D^2 = K_T\cdot D = -1$ and the intersection form of $D$ is negative definite.
For all such divisor, there exists a morphism $f:T \to T'$ onto a
smooth, projective surface $T'$, which contracts $D$ to a point and is
an isomorphism on the complement of $D$.

\subsection {Simple internal projections}\label{ssec:sip} Let $S \subset \bP^N$ be a degree $d$ surface, of sectional genus $g$.
We call
\emph{simple internal projection} of $S$
a surface $T\subset \bP^{N'}$ obtained by projecting $S$ from a
curvilinear subscheme $Z$ of length $b$
supported on the smooth locus of $S$, with    
$N'=N-\dim (\langle Z\rangle)-1$, such that the intersection scheme of $\langle Z\rangle$ with $S$ is $Z$ and such that the projection map
is birational.
We recall that a scheme $Z$ is \emph{curvilinear} if for every point $p$
in the support of $Z$, the Zariski tangent space of $Z$ at $p$ has
dimension at most one.

For $T$ a simple internal projection of $S$ as above,
one has $\deg(T) = d-b$ and $T$ has the
same sectional genus $g$ as $S$.
Note that, if $d-b \geq 2g+1$ and $S$ is regular and linearly normal, 
then any projection from a
curvilinear subscheme $Z$ of length $b$ supported on the smooth locus
of $S$ is a simple internal projection,
and $N'=N-b$,
because in this case the linear system of hyperplane sections of $S$
passing through $Z$ restricts on its general member to a complete,
non-special, very ample linear system.

  In conclusion, note that $T\subset \bP^{n}$ is a simple internal projection if there exists a surface $S\subset \bP^{n+1}$ such that $T$ is the projection of $S$    from  a smooth point $p\in S$ and the projection is birational.

\section{Some preliminary results} 
\label{S:prelim2}

\noindent
We keep the setup and notation as in Section \ref {ssec:setup}.

\begin{lem}
\label{l:g->n}
Let $S$ be a degree $d$, irreducible and non-degenerate surface in
$\P^n$ with sectional genus $g$. 
If $d>3g-3$ and $S$ has linearly normal hyperplane sections, then
$d < \frac 3 2 n$.
\end{lem}

\begin{proof}
One has $3g-2 \geq 2g-1$ if $g\geq 1$, hence the divisor $\restr H C$
is non-special for all $g$.
Since,  by the hypotheses,   the hyperplane sections of $S$ are linearly normal, one has
$n = d-g+1$ by Riemann--Roch.
Then
\[
  d>3g-3
  \iff
  d> 3(d-n),
\]
which proves the assertion.
\end{proof}

The rays generated by $H$ and $K_{S'}$ in the cone of divisors of $S'$
lie respectively inside and outside of the cone of effective divisors,
since $H$ is effective and no multiple of $K_{S'}$ is effective,  because in our hypotheses, $S'$ has Kodaira dimension $-\infty$.  
When we let the integer $m$ move to infinity, the ray generated by
$mK_{S'}+H$ travels in straight line between the two rays generated by
$H$ and $K_{S'}$ respectively; therefore, if $m$ is large enough, then 
no multiple of $mK_{S'}+H$ has sections. More  precisely  we have:

\begin{lem}\label{lem:ad2_intro}
Suppose that
$d > \frac {2m_0}{m_0-1}(g-1)$ for some integer $m_0$.
Then for all $m\geq m_0$, no multiple of $mK_{S'}+H$ has sections.
In particular, if $d>3g-3$, then for all $m\geq 3$, no multiple of
$mK_{S'}+H$ has sections.
\end{lem}

\begin{proof}
Let $C$  be a general curve in $|H|$. It is an irreducible curve, and $C^2=d \geq 0$.
Therefore, if $(mK_{S'}+H)\cdot C <0$,
then no multiple of $mK_{S'}+H$ has sections,
see \cite[Useful Remark III.5]{beauville} for instance.
Now,
\[
  (mK_{S'}+H)\cdot C
  = m(K_{S'}+H)\cdot H -(m-1) H^2
  = m(2g-2)-d(m-1),
\]
which is negative if $m \geq m_0$.
\end{proof}

In  Section \ref{sec:nonempt} below, 
 we classify rational surfaces under the 
assumption that the tri-adjoint linear system $|3K_{S'}+H|$ is
empty.  By Lemma  \ref {lem:ad2_intro}, this is certainly the case if $d>3g-3$.  The general idea is to use that this implies that, since $S$ is
rational, any irreducible curve in $|2K_{S'}+H|$ is rational.

If $d=3g-3$, it is possible that $3K_{S'}+H$ has sections. An example
is provided by the Veronese surface $V_9 \subset \P^{54}$:
then $S'=\P^2$ is the projective plane, and $H=9L$ where $L$ is the
class of lines, so that $d=81$ and $g = 28$.

\begin{prop}\label{qd}
Let $S, S',H$ be as  in Section \ref {ssec:setup}.
The following facts hold: \\
\begin{inparaenum}[(i)]
\item\label{qd-lem:adj}
  $h^0(S', K_{S'}+H)=g-q$; in particular, $q\leq g$, and
  $h^0(S', K_{S'}+H)=0$ if and only if $g=q$;
\item\label{qd-i}
 $h^1(S',H)\leq q$;\\
\item\label{qd-ii}
  $g=a+1-q+h^1(S', H)\leq a+1$;\\
\item\label{qd-iii}
 if $g=q$ then $a+1= 2q-h^1(S',H)$, and in particular $q\leq
  a+1\leq 2q$; \\
\item\label{qd-iv}
 if $q=a+1$ then $g=q=h^1(S', H)$; \\
\item\label{qd-v}
 the hyperplane sections of $S$ are linearly normal if and only if   $q=h^1(S',H)$, and this happens if and only if $g=a+1$;  \\
\item\label{qd-vi}
 if $q=0$, then $g=a+1$ and the hyperplane sections of $S$ are
 linearly normal;
 if moreover $d\geq 2g+1$, then  $S$ has only isolated singularities;\\
\item\label{qd-viii}
 $(K_{S'}+H)^2=K_{S'}^2+4g-4-d$;\\
\item\label{qd-ix}
  $(K_{S'}+H)^2\leq K_{S'}^2-n+3a$, and if  $g=a+1$ (i.e., if the hyperplane sections of $S$ are linearly
  normal),  in particular if $q=0$,
  then equality holds.
\end{inparaenum}
\end{prop}

\begin{rem}
In the above proposition, the assumption that $d\geq 2g-1$ may be
substituted with the assumptions that $S$ has Kodaira dimension
$-\infty$ and $h^1(C,  \restr H C)=0$.
\end{rem}

\begin{proof}   
Let $C\in |H|$ be a smooth curve.
Since $H$ is big and nef, one has $h^1(S', K_{S'}+H)=0$
by Kawamata-Viehweg vanishing.
Then the long exact sequence obtained from 
\begin{equation*}
  0\to\mathcal O_{S'}(K_{S'})\to
  \mathcal O_{S'}(K_{S'}+H)\to \mathcal \O_C(K_C) \to 0.
\end{equation*}
gives $g=h^0(S', K_{S'}+H)+q$, which proves \eqref{qd-lem:adj}.

One has $h^1(C,\restr H C)=0$ hence, from the long exact sequence
obtained from  
\begin{equation}\label{eq:wet}
 0\to\mathcal O_{S'}\to  \O_{S'}(H) \to \O_C(\restr H C)\to 0,
\end{equation}
we have  $h^1(S',H)\leq q$, proving \eqref{qd-i}.


Since $S$ is linearly normal,  one has  $h^0(S', H)=n+1$, and so,
again by \eqref {eq:wet}, one has    $h^0(C,\restr H
C)=n+q-h^1(S',H)\leq n+q$.  On the other hand, since $h^1(C,\restr H
C)=0$,  we have $h^0(C,\restr H C)=n+a-g+1$. Comparing both expressions
of $h^0(C,\restr H C)$, we obtain $g= a+1 - q+h^1(S',H)\leq a+1$,
proving \eqref{qd-ii}  (cf. \cite[Prop. 1.1]{how}).
Then \eqref{qd-iii} follows right
away. As for \eqref{qd-iv}, if $q=a+1$, then from
\eqref{qd-lem:adj}, \eqref{qd-i}, and \eqref{qd-ii},
one has $h^1(S',H)\leq q\leq g=h^1(S',H)$.

The hyperplane sections of $S$ are linearly normal if and only if
the restriction map $H^0(S', H) \to H^0(C, \restr H C)$ is
surjective. By  \eqref {eq:wet}, and since $h^1(C,\restr H C)=0$, this
happens if and only if $h^1(S',H)= q$, which is equivalent to $g=a+1$
by \eqref{qd-ii},  proving \eqref{qd-v}.

In particular, if $q=0$, then $h^1(S',H)=0$  by \eqref{qd-i},   and by \eqref{qd-v} the hyperplane
sections of $S$ are linearly normal, and $g=a+1$.
If $d\geq 2g+1$, then $\restr H C$ is very ample, hence
the general hyperplane section of $S$ is smooth,
and so
$S$ has only isolated singularities, which proves \eqref{qd-vi}.


Then $(K_{S'}+H)^2=K_{S'}^2+4g-4-2n-2a+n+a=K_{S'}^2+4g-4-n-a$, proving
\eqref{qd-viii}.
Finally, since  $g-1\leq a$, also \eqref{qd-ix} follows.
\end{proof}

\begin{lem}[{\cite[Lemma 2.6]{m}}]
\label{l:1-conn}
Let $T$ be a smooth, irreducible, projective surface,
and let $L$ be a big and nef divisor on $T$. Then all divisors in
$|L|$ are $1$-connected, and thus have non-negative arithmetic genus.
  
If $L =A+B$ is a decomposition of $L$ with $A$ and $B$ effective
divisors such that $A \cdot B = 1$ and $A^2\leq B^2$, then $A$ and $B$
are 1-connected, and only the following possibilities can occur:
\begin{compactenum}[(a)]
\item
  $A^{2} =-1$ (and so $L\cdot A=0$);
\item
  $A^{2} =0$ (and so $L\cdot A=1$);
\item
$A^{2}=B^{2}=1$, $A \equiv B$, and $L^{2}=4$.
\end{compactenum}
\end{lem}

\begin{lem}[{\cite[Lemma (A.4)]{cfm}}]
\label{l:CFM}
  Let $T$ be a smooth, irreducible, projective surface.
  Let $D$ be an $m$-connected effective divisor on $T$. Let $D_1,D_2$
  be effective divisors such that $D=D_1+D_2$.
  \\(i) If $D_1 \cdot D_2 = m$, then $D_1$ and $D_2$ are
  $\lfloor{(m+1)/2}\rfloor$-connected.
  \\(ii) If $D_1$ is minimal with respect to the condition $D_1\cdot
  (D-D_1)=m$, then $D_1$ is
  $\lfloor{(m+3)/2}\rfloor$-connected.
\end{lem}

\begin{lem}
\label{reider-1con}
Let $T$ be a smooth, irreducible, projective surface, 
equipped with a big and nef divisor $L$ satisfying $L^2\geq 3$.
Let $D$ be a non-zero, effective, divisor, such that $D^2=0$
and $L\cdot D=1$.
Then $D$ is 1-connected. 

Moreover, if $D$ decomposes as $D=A+B$ with $A$ and $B$ effective
divisors such that $A\cdot B= 1$, then $A^2=B^2=-1$ and, possibly up to
exchanging $A$ and $B$, 
$L\cdot A=0$ and $L\cdot B=1$.

Assume in addition that $L=K_T+L_0$ with $L_0$ a big and nef divisor, and
$D$ is not 2-connected.
Then for every decomposition $D=A+B$ as above with $L\cdot A=0$,
one has, $K_T\cdot A=-1$, hence $A$ is a $(-1)$-divisor,
and $L_0\cdot A=1$.
\end{lem}

\begin{proof} 
First note that for any decomposition of $D=A+B$ with $A, B$ effective
divisors, we may assume without loss of generality that $L\cdot A=0$
and $L\cdot B=1$, because $L$ is nef and $L\cdot D=1$.
Then the index theorem implies that $A^2<0$ and $B^2\leq 0$.

Thus, the equality $D^2=A^2+2A\cdot B+B^2=0$ implies that $A\cdot B\geq
1$, hence $D$ is 1-connected.
   
Suppose now in addition that $A\cdot B=1$.
Then again from $D^2=A^2+2A\cdot B+B^2=0$, we have either
$A^2=-1$ and $B^2=-1$, and then the decomposition is as asserted in
the lemma,
or $A^2=-2$ and $B^2=0$.
This second possibility, however, cannot occur,
for otherwise we would have $(D+B)^2=2$, $L\cdot (D+B)=2$, and
$L^2\geq 3$, in contradiction with the index theorem.
This proves the second assertion of the lemma.
   
Assume, finally, that $L=K_T+L_0$ with $L_0$ a big and nef
divisor. 
Since $A$ is 1-connected, by Lemma~\ref{l:CFM},
and $A^2=-1$,
one has $K_T\cdot A = 2p_a(A)-1 \geq -1$.
Then $(K_T+L_0)\cdot A=0$ gives
$L_0 \cdot A = -2p_a(A)+1$.
Since $L_0$ is nef, this implies that $p_a(A)=0$, and 
$K_T\cdot A = -1$ and $L_0 \cdot A = 1$.
\end{proof}

\begin{prop}
\label{composto}
Let $T$ be a smooth, irreducible, projective surface,
and $D$ be a non-zero, effective divisor on $T$, such that $D$ is nef
and $D^2=0$.
Then:
\\(i) for all effective $A$ such that $A\leq D$, one has $D\cdot A=0$
and $A^2 \leq 0$;
\\ (ii) $D$ is $0$-connected.

If moreover $D=K_T+L$ with $L$ nef and big,
and $T$ has negative Kodaira dimension,
then there exists a (possibly
non-linear) base-point-free pencil $\{G\}$ of rational curves, such
that $D$ consists 
of members of $\{G\}$.
\end{prop}

\begin{proof}
Let $A$ be an effective divisor such that $A\leq D$. Then there exists
an effective $B$ such that $D=A+B$. One has $D^2 = D\cdot A+D\cdot
B=0$,
and $D\cdot A$ and $D\cdot B$ are non-negative; therefore 
$D\cdot A = D\cdot B=0$. Then $A^2>0$ would contradict the index
theorem. This proves (i).

Consider a decomposition $D=A+B$ with $A$ and $B$ non-zero and effective.
Since $D^2 = A^2+B^2+2A\cdot B =0$, if  $A\cdot B <0$, then either
$A^2>0$ or $B^2>0$, which is impossible by (i). Thus $D$ is
$0$-connected.

Assume now that $D=K_T+L$ with $L$ nef and big, and $T$ has negative
Kodaira dimension.
Assume first that $D$ is $1$-connected. One has
$D^2 = D\cdot K_T + D\cdot L=0$, and $D\cdot L \geq 0$, hence
$D\cdot K_T \leq 0$. Since $D^2=0$ and $p_a(D) \geq 0$, one must have 
$D\cdot K_T = -2$ and $p_a(D) = 0$.
Then either $q=0$, and then $D$ moves in a linear base-point-free
pencil of rational curves, as required, or $q>0$, and then $D$ is a
fiber of the Albanese pencil.

If $D$ is not $1$-connected, since it is $0$-connected, there exists a
decomposition $D=A+B$ with $A$ and $B$ non-zero and effective, and
$A\cdot B = 0$. Assume $A$ is minimal with respect to this condition.
Then by Lemma~\ref{l:CFM}, $A$ is $1$-connected and $p_a(A) \geq 0$.
One has $A\cdot D = A^2 + A\cdot B = A^2$.
Since $A\cdot D = 0$ by (i), one has $A^2=0$.
By the index theorem, this implies that $A\cdot L >0$.
Now,
$A\cdot K_T + A\cdot L = 0$, and therefore $A\cdot K_T < 0$,
which implies that $p_a(A) = 0$. Then, as above, either $q=0$ and $A$
moves in a base-point-free pencil, or $q>0$ and $A$ is a fiber of the
Albanese pencil. Finally, since $A\cdot B=0$ and $B^2=0$, $B$ is a
union of members of the pencil $\{A\}$.
\end{proof}

\section{Vanishing of the first adjoint}
\label{S:vanish-adj1}

\begin{prop} \label{minimal}
Let $S,S',H,C$ be as in Section~\ref{ssec:setup},
and assume that $h^0(S',K_{S'}+H)=0$
(equivalently, $g=q$, by Proposition~\ref{qd}, \eqref{qd-lem:adj}).
Then: \\
\begin{inparaenum}[(i)]
\item  $h^0(S',K_{S'}+2H) =d+q-1>0$;\\
\item $K_{S'}+2H$ is nef;\\
\item $K_{S'}^2=8-8q$ and $(K_{S'}+2H)^2=0$,
  unless  $S$ is the Veronese surface $V_2$ of degree $4$ in $\bP^5$,
  in which case $K_{S'}^2=9$ and $(K_{S'}+2H)^2=1$;\\
\item  either $S=V_2$, or  $S$ is a scroll. 
\end{inparaenum}
\end{prop}

\begin{proof}  Since  $h^0(S', K_{S'}+H)=0$ by the hypotheses,
and   $h^1(S', K_{S'}+H)=0$ by the Kawamata--Viehweg
theorem, one has $h^0(S',K_{S'}+2H)=h^0(C, \restr[K_{S'}+2H] C)$.
Note that $(K_{S'}+2H)\cdot C=2g-2+d$, so that
$h^1(C, \restr[K_{S'}+2H] C)=0$.  
Applying Riemann--Roch we get  
$$
h^0(S',K_{S'}+2H)=(K_{S'}+2H)\cdot C-g+1=2g-2+d-g+1=d+g-1=d+q-1. 
$$
Moreover $d\geq n-1 \geq 2$, so $d+q-1>0$, and (i) follows. 

(ii) Assume that $K_{S'}+2H$ is not nef. If $\theta$ is  an
irreducible curve  such that $(K_{S'}+2H)\cdot \theta<0$, then
$\theta^2<0$ and, since $2H$ is nef, we must have $K_{S'} \cdot
\theta<0$. So the 
only possibility is that $\theta$ is a $(-1)$-curve contracted by $H$, contradicting the assumption that $S'$ is a minimal desingularization of $S$. 

(iii) We have $(K_{S'}+2H)^2\geq 0$.
Since $H^2=d$ and $K_{S'} \cdot H=2g-2-d=2q-2-d$,  we have
$(K_{S'}+2H)^2= K_{S'}^2+8q-8$, and so $K_{S'}^2\geq 8-8q$. Since the
Kodaira dimension of $S'$ is $-\infty$, we get $K_{S'}^2=8-8q$, unless
$S'$ is isomorphic to $\bP^2$.  
In this case we have $(K_{S'}+2H)^2=1$, which implies that  $K_{S'}+2H\sim L$ where $L$ is a line, i.e.,  $H\sim 2L$ and so $S=V_2$.

(iv) Assume $S\neq V_2$. Then $K_{S'}^2=8-8q$ and $(K_{S'}+2H)^2=0$. 
Then, by Proposition~\ref{composto} applied to $L=2H$, 
$|K_{S'}+2H|$ is composed with a
(possibly non-linear) pencil $\{G\}$ of rational curves with $G^2=0$.
Thus, $K_{S'}\cdot G = -2$. 
Since  $(K_{S'}+2H)\cdot G = 0$, we must have $H\cdot G=1$, so that
$S$ is a scroll.
\end{proof}
 
\begin{rem}\label{rem:obs}
In the situation of Proposition \ref {minimal}, suppose that $S\neq V_2$, so that $K_{S'}^2=8-8q$ and $(K_{S'}+2H)^2=0$.
If $q=0$, one has $K_{S'}^2=8$, and $S'\cong \bF_d$ is a  Hirzebruch surface. If $q>0$, then $S'$ is a minimal ruled surface of genus $q$.

If $q=0$, one has $a=-1$
(as $q\leq a+1 \leq 2q$, see Proposition \ref {qd}, \eqref{qd-iii})
hence $d=n-1$, so $S\neq V_2$ is a surface of minimal degree that is a scroll.  Precisely the pencil $\{G\}$ is the  ruling $|F|$ of $S'=\bF_e$ and  $K_{\bF_e}+2H\sim (d-2) F$. Since $K_{\bF_e}\sim -2E- (e+2)F$ we obtain  
$$2H\sim 2E+ (d+e) F,\quad \text{ i.e.,}\quad  H\sim E+\frac {d+e}2 F,$$
in particular $d+e$ must be even. 

Note that we must have $H\cdot E\geq 0$, which is equivalent to $d\geq e$. If $d=e$ then $S$ is the cone over a rational normal curve, while if $d>e$, $H$ is very ample and $S=S'$. 

If $q>0$ then the pencil $\{G\}$ is the Albanese pencil, and from
$h^0(S,K_{S'}+2H) =d+q-1>q$ we obtain $K_{S'}+2H\equiv (n+a+2q-2) G$.
Indeed, let $\alb: S' \to \Gamma$ be the Albanese fibration;
$K_{S'}+2H$ is the pull-back by $\alb$ of a divisor $D$ on $\Gamma$
such that $H^0(S,K_{S'}+2H)=H^0(\Gamma,D)$. Having $h^0(\Gamma,D)>q$,
$D$ is non-special, hence by Riemann--Roch it has degree $d+2q-2$.
\end{rem}

\noindent
The following lemma essentially gives the converse to Proposition \ref {minimal}, (iv):

\begin{lem}\label{lem:stup}
If $S$ is a scroll, then   $h^0(S' ,K_{S'}+H)=0$  and $g=q$. 
\end{lem}

\begin{proof}
Suppose that $S$ is a scroll and that $h^0(S' ,K_{S'}+H)>0$.
Denote  by $L$  the pull back to $S'$ of a general line of the ruling of $S$.  Then $H\cdot  L=1$ and  $L^2= 0$.  Since $L$ is a rational curve,  the adjunction formula  implies  $K_{S'}\cdot  L= -2$.    Then $(K_{S'}+H) \cdot L<0$, a contradiction  because $L$ is  nef. This proves the first assertion. The second assertion follows from Proposition~\ref{qd}, \eqref{qd-lem:adj}.  
\end{proof}

\begin{rem}\label{rem:obs2}
Let $S\subseteq \bP^n$ be a linearly normal, non--degenerate {rational}  surface of degree $d$. If $S$ is a scroll, then by Remark \ref {rem:obs}, $S$ is a surface of minimal degree $d=n-1$. In particular, if $d\geq n$, then $S$ is not a scroll. Conversely, if $d=n-1$, then $S$ is a surface of minimal degree and therefore it is either the Veronese surface $V_2$ or a scroll. 
\end{rem}

\begin{cor}\label{cor:top}
  Let $S,S',H,C$ be as in Section~\ref{ssec:setup},
and assume that $q>0$ and $h^0(S',K_{S'}+H)=0$.
Then $S$ is a scroll, and:\\
\begin{inparaenum}[(i)]
\item  one has $d\geq n+q-1$; if equality holds, then $S$ is a cone
  over a curve of geometric genus $q$;\\  
\item one has $d < n+2q$.
\end{inparaenum}
\end{cor}

\begin{proof} That $S$ is a scroll is contained in
  Proposition~\ref{minimal}. 

(i) By Proposition \ref {qd}, \eqref{qd-ii}, one has $q=g\leq a+1$, hence
$d=n+a\geq n+q-1$. If equality holds, then $h^1(S',H)=q=g$ (see
Proposition \ref {qd}, \eqref{qd-iv}), and this occurs if and only if
the hyperplane sections of $S$ are linearly normal (see Proposition
\ref {qd}, \eqref{qd-v}).
Then the assertion follows by C. Segre's theorem \cite[Thm 2.3]{CD}.

(ii) By Proposition \ref {qd}, \eqref{qd-iii}, we have 
$a= 2q-1-h^1(S',H)$, thus $d=n+a=n+2q-1-h^1(S',H)\leq n+2q-1$, proving the assertion.
\end{proof}

\section{General properties of the first adjoint system} \label{sec:genprop}

In this section we want to discuss the first adjoint system when
$h^0(S', K_{S'}+H) >0$. By Proposition~\ref{qd}, \eqref{qd-lem:adj}, this assumption
is equivalent to $g>q$.

\subsection{Nefness of the adjoint system}
\label{s:4.2}

\begin{prop}\label{nef}  Let $S, S'$ and $H$ be as in Section \ref
  {ssec:setup}, and assume that $h^0(S', K_{S'}+H) >0$.
Then the following assertions hold:\\ 
\begin{inparaenum}[(i)]
\item  $K_{S'}+H$ is nef, hence $(K_{S'}+H)^2\geq 0$; \\
\item  $K_{S'}^2-n+3a\geq 0$;\\
\item  $d\leq 4g-8q+4$, unless $S$ is
  the Veronese surface $V_3$ of degree $9$ in $\bP^9$ ($d=9$, $g=1$); \\
\item $3a\geq n+8q-8$, unless $S$ is the Veronese surface $V_3$ of degree $9$. 
\end{inparaenum}
\end{prop}

\begin{proof} 
(i) Assume that $K_{S'}+H$ is not nef, and let $\theta$ be an irreducible curve such that
$(K_{S'}+H)\cdot \theta<0$. Since  $H$ is nef, one has $K_{S'} \cdot
\theta<0$. On the other hand, since $K_{S'}+H$ is effective,
$(K_{S'}+H)\cdot \theta<0$ yields  $\theta^2<0$.
This implies that $\theta$ is a $(-1)$-curve and $H\cdot  \theta =0$,
which is impossible because $S'$ is a minimal desingularization of $S$. 

(ii) By Proposition \ref{qd},  \eqref{qd-ix}, we have
$0\leq (K_{S'}+H)^2\leq
K_{S'}^2-n+3a$,  as required.

(iii)--(iv)
By Proposition \ref{qd}, \eqref{qd-viii} and \eqref{qd-ix}, we have
\begin{equation}
\label{eq:Kcarre}
  d = 4g-4+K_{S'}^2-(K_{S'}+H)^2
  \quad
  \text{and}
  \quad
  3a-n \geq (K_{S'}+H)^2 - K_{S'}^2.
\end{equation}
Since $S'$ has negative Kodaira dimension,
either $K_{S'}^2 \leq 8-8q$,
or $K_{S'}^2 = 9$ and $S'$ is isomorphic to $\P^2$.
In the former case, we obtain the required inequalities at once,
using that $(K_{S'}+H)^2\geq 0$.
In the latter case, $q=0$, and the inequality in \eqref{eq:Kcarre}
is an equality by Proposition \ref{qd}, \eqref{qd-ix}; thus the two
inequalities in (iii) and (iv) hold unless
$(K_{S'}+H)^2=0$.
If $(K_{S'}+H)^2=0$, then $K_{S'}+H\sim 0$,
because $\bP^2$ does not contain effective
non-zero divisors with self-intersection $0$,
so $H\lineq -K_{S'}$,
and $S$ must be the Veronese surface of degree $9$ in $\bP^9$.
\end{proof}

\begin{prop}\label{feixe}  Let $S, S'$ and $H$ be as in Section \ref
  {ssec:setup} and suppose that $h^0(S', K_{S'}+H)>0$ and
  $(K_{S'}+H)^2=0$.
Then either
\\ (a) $K_{S'}+H\sim 0$ and $S'$ is a Del Pezzo or
weak Del Pezzo surface (and $g=1$), or
\\ (b) $g\geq 2$, and
$|K_{S'}+H| = |(g-1) G|$ where $\{G\}$ is a pencil of rational
curves (and so, the Albanese pencil if $q>0$), satisfying $G\cdot
H=2$; thus $S$ is ruled by conics.
\end{prop}

\begin{proof}  
If $K_{S'}+H\sim 0$,
then $S'$ is a 
Del Pezzo or weak Del Pezzo surface. 
 
Assume  $K_{S'}+H\not\sim 0$. Then, since $K_{S'}+H$ is nef and
$(K_{S'}+H)^2=0$, it follows from Proposition~\ref{composto}
that $|K_{S'}+H|$ is composed with a base point free pencil $\{G\}$ of
rational curves. 
Then $G\cdot H=2$ because $G\cdot (K_{S'}+H)=0$.
\end{proof}

In the remainder of this section we will consider the situation when 
$h^0(S', K_{S'}+H)>0$ and
$(K_{S'}+H)^2>0$.
The following lemma implies in particular that in this situation $S$
is not ruled by conics.

\begin{lem}\label{grau}
Let $S, S'$ and $H$ be as in Section \ref {ssec:setup}.
If $h^0(S', mK_{S'}+H)>0$ for some integer $m$, then any pencil of
rational curves $\{G\}$ with self-intersection  $0$ (that always
exists if $S'\neq \mathbb P^2$) satisfies $H\cdot G \geq 2m$.
Furthermore, if equality holds, then either  $mK_{S'}+H\lineq 0$,
or   
 all the components of any effective divisor in $|mK_{S'}+H|$ are
 contained in fibers of $\{G\}$. 

\end{lem} 
\begin{proof}   Since $G$ is nef,  $G\cdot (mK_{S'}+H)\geq 0$.  From  $K_{S'}\cdot G=-2$,  we obtain $H\cdot G\geq 2m$.  

If  $H \cdot G= 2m$, then $G\cdot (mK_{S'}+H)=0$, hence either
$mK_{S'}+H\lineq 0$
or, since $G$ is nef, any component of an effective divisor in
$|mK_{S'}+H|$ is contained in a fiber of  $\{G\}$. 
\end{proof}

\subsection{Base points of the adjoint system}
\label{s:4.3}


\begin{prop}\label{connected}
\label{prop:bpf}
Let $S,S'$ and $H$ be as in Section \ref {ssec:setup},
and suppose that $h^0(S', K_{S'}+H)>0$ and $(K_{S'}+H)^2>0$. 
If $|K_{S'}+H| $ has base points, then $q>0$, and there exists
a smooth irreducible curve $\theta$ of genus $q$ such that
$\theta^2=-1$ and $H\cdot \theta=0$.
\end{prop}

\begin{proof}
Note first that the hypotheses imply $g\geq 2$, and thus $H^2\geq 4$.
In fact $(K_{S'}+H)^2>0$ implies,  by the index theorem,
that $(K_{S'}+H)\cdot H>0$, and so $g\geq 2$.

Let $x$ be a base point of  $|K_{S'}+H|$. Since $h^0(S', H)\geq 4$,
$x$ is a multiple point of some curve $C$ in $|H|$ and so, by Theorem
3.1 of \cite {m},  $C$ is not 2-connected.
Then $C$  has  a decomposition  $C=A+B$ with
$A,B>0$, $A\cdot B=1$, and $A^2\leq B^2$;
by Lemma~\ref{l:1-conn}, $A$ and $B$ are 1-connected
and only the following possibilities can occur:\\
\begin{inparaenum}[(a)]
\item\label{5.4a}
  $A^2=B^2=1$, $A\equiv B$, $H^2=4$ and $H\equiv 2A$ or\\
\item\label{5.4b}
  $A^2=-1$, $H\cdot A=0$, or \\
\item\label{5.4c}
  $A^2=0$, $H\cdot A=1$.
\end{inparaenum}

In case \eqref{5.4a}, from $K_{S'}\cdot H<0$ and $p_a(H)\geq 2$, one has
$K_{S'}\cdot H=-2$.  Then the hypothesis $(K_{S'}+H)^2>0$ implies
$K_{S'}^2\geq 1$, and so $S'$ is a rational surface, hence $H\sim 2A$.
On the other hand $A \cdot (A+K_{S'})=0$ and $(A+K_{S'})^2\geq 0$
yield by the index theorem that $A\sim -K_{S'}$.
By the Riemann-Roch theorem, $h^0(S', A)\geq 2$.  Since $p_a(A)=1$ and
$H\cdot A=2$, the map defined by $|H|$ restricted to a general member
of $|A|$ is not birational, contradicting our assumptions on $H$. 
So case \eqref{5.4a} does not occur.

In case \eqref{5.4b}, we claim that $p_a(A)\geq 1$. It suffices to prove that
$K_{S'}\cdot A\geq 1$. By the adjunction
formula, $K_{S'}\cdot A$ is odd, and so, since $A$ is 1-connected,
one has $K_{S'}\cdot A\geq -1$.  But
$K_{S'}\cdot A=-1$ and $H\cdot A=0$ cannot occur, because in that case
$A\cdot (K_{S'}+H)=-1$, in contradiction with the fact that $K_{S'}+H$
is nef. This proves the claim.

If $q=0$, then $h^0(K_{S'}+A)> 0$,
which leads to a contradiction, since
$(K_{S'}+A)\cdot H = K_{S'}\cdot H <0$
 and $H$ is nef and effective.

The upshot is that $q>0$ and $p_a(A)\geq 1$.
Then $A$ is not contained in the fibres of the Albanese map of $S'$,
and so $p_a(A)\geq q$.  
If $p_a(A)> q$, then  $h^0(K_{S'}+A)> 0$, as follows from the exact sequence
$$
0\to \mathcal O_{S'}(K_{S'})\to \mathcal O_{S'}(K_{S'}+A)\to \omega_A\to 0.
$$ 
This leads to a contradiction as above, since $H\cdot (K_{S'}+A)<0$.

Thus $p_a(A)=q>0$.  Since $A^2=-1$, one has $K_{S'}\cdot A=2q-1$.
Now $A$, having arithmetic genus $q>0$,   cannot be contained in a
fibre of the Albanese map of $S'$, and so it has at least one component
$\theta$ with geometric genus $\gamma\geq q$.  Since $q\leq \gamma\leq
p_a(\theta)\leq q$, the curve $\theta$ is smooth of genus $q$. 

Since $H$ is nef and $H\cdot A=0$, one has $H\cdot \theta=0$, hence
also $\theta^2<0$.   On the other hand  $K_{S'}+H$ being  nef and
$(K_{S'}+H) \cdot A=2q-1$ imply that  $(K_{S'}+H)\cdot \theta \leq
2q-1$.  Hence, from the adjunction formula, we obtain $\theta^2=-1$,
and the property asserted in the statement holds.

Lastly, assume we are in case \eqref{5.4c}.
Then $H\cdot A=1$ so, since $K_{S'}+H$ is nef, $K_{S'}\cdot
A\geq -1$ and $p_a(A)\geq 1$.  Since $H\cdot A=1$ and  $H$ is base
point free, $h^0(A, H)\geq 2$. Then, by \cite[Prop. A.5, (ii)]{cfm},
$A$ is not 2-connected. So $A$ has a decomposition $A=A_1+A_2$ with
$A_1\cdot A_2=1$ and,
by Lemma~\ref{reider-1con},
we have $A_1^2=A_2^2=-1$ and
$H\cdot A_1=0$ and $H\cdot A_2=1$, possibly up to
relabelling $A_1$ and $A_2$.
Moreover, $A_1$ is $1$-connected by Lemma~\ref{l:CFM},
so we are back in case \eqref{5.4b} with $A_1$ instead of $A$, and the
existence of the curve $\theta$ follows by the argument given in that
case. 
\end{proof}

\begin{rem}
Conversely, the existence of a curve $\theta$ as in
Proposition~\ref{prop:bpf} tends to impose that the adjoint system has
base points.
Indeed, in the situation of the proposition, 
consider $\hat H = H-\theta$; it is an effective divisor, and
\[
  \restr[K_{S'}+H] \theta
  = \restr[K_{S'}+\theta] \theta + \restr {\smash{\hat H}} \theta
  = K_\theta + \restr {\smash{\hat H}} \theta.
\]
Since $\hat H \cdot \theta = H\cdot \theta -\theta^2=1$,
the linear series $|K_\theta + \restr {\smash{\hat H}} \theta|$ on
$\theta$ has a base point if $\restr {\smash{\hat H}} \theta$ is
effective, and in that case the adjoint system
$|K_{S'}+H|$ has a base point as well.
Note that $\restr {\smash{\hat H}} \theta$ is certainly
effective if $q=1$, or if $|\hat H|$ doesn't have $\theta$
as a fixed component.
\end{rem}

\begin{prop}
\label{p:bpf-rat}
Let $S,S'$ and $H$ be as in Section \ref {ssec:setup},
and suppose that $h^0(S', K_{S'}+H)>0$, and $q=0$.
Then $|K_{S'}+H| $ is base-point-free.
\end{prop}

\begin{proof}
If $(K_{S'}+H)^2>0$, the conclusion follows from
Proposition~\ref{connected}.
Otherwise, since $K_{S'}+H$ is nef, one has $(K_{S'}+H)^2=0$.
Then, by Proposition~\ref{composto}, there exists a base-point-free pencil $\{G\}$ of rational curves, such
that $D$ consists  of members of $\{G\}$. Since $S'$ is
rational, $h^0(S',G) >1$, and the conclusion follows.
\end{proof}

\subsection{Examples}
\label{s:examples}
\begingroup
\def\p#1{p^{(#1)}}

In this section we present some examples illustrating
Proposition~\ref{prop:bpf} above. 
They show that it is indeed possible that the adjoint system has base
points, and then these base points are located on the curve $\theta$
as in the proposition.
Example~\ref{ex:ell_ruled_e=1} satisfies all the assumptions of the
proposition, in particular $d \geq 2g-1$.
In Examples~\ref{ex:ell_ruled_det0-dec} and
\ref{ex:ell_ruled_det0-indec} however, one has $d=2g-2$; still, the two
latter examples have many interesting features.
All our examples live in elliptic ruled surfaces. It is possible to
cook up similar examples on ruled surfaces of irregularity $q>1$, but
then $d$ becomes smaller than $2g-1$.

Let $\Gamma$ be a smooth, irreducible, elliptic curve.
We shall consider elliptic ruled surfaces
$R=\P(\mathcal{E})$ for various rank two vector bundles $\mathcal{E}$
on $\Gamma$, which we assume to be normalized as in
\cite[V, Notation~2.8.1]{hartshorne}.
In each case, we fix a section $C_0$ such that
$\O_R(C_0) \cong \O_{\P(\mathcal{E})}(1)$.
For each divisor $\fd$ on $\Gamma$, we denote by $F(\fd)$ the pull-back
of $\fd$ to $R$; we denote by $F$ the numerical equivalence class of
$F(\fd)$ for all degree $1$ divisors $\fd$ on $\Gamma$.
One has
\(
K_R \lineq -2C_0 + F(\det \mathcal{E})
\).

\begin{expl}
\label{ex:ell_ruled_e=1}
Fix a point $p\in \Gamma$ and
consider
$R = \P(\O_\Gamma\oplus \O_\Gamma(-p))$, with the notation introduced
above for all of Section~\ref{s:examples}.
For all $k\in\N$, let
\[
  H_k \lineq k\bigl(C_0+F(p)\bigr).
\]
The curve $C_0$ has genus $1$ and self-intersection $-1$, and
intersects $H_k$ with degree $0$.
We shall prove the following.

\begin{claim}
For all $k\geq 3$,
the linear system $|H_k|$ defines a birational morphism
which does not contract any $(-1)$-curve, onto a surface
$S \subset \P^{\frac 1 2k(k+1)}$ of degree $d=k^2$
and sectional genus $g = \frac 1 2 k(k-1)+1$,
and with linearly normal hyperplane sections.
The adjoint system $|K_R+H_k|$ has a base point on $C_0$.
\end{claim}

Before we prove the claim, let us observe that
\[
  d = H_k^2= 2g-2 +k.
\]
Thus, we get examples fitting in the assumptions of
Proposition~\ref{prop:bpf}, with degree $d$ arbitrarily large with
respect to the genus $g$.
Since $S$ has linearly normal hyperplane sections, results by C.~Segre
(mentioned in the introduction) and Hartshorne (see \cite {Ha}) imply
that the $d\leq 4g+4$, see \cite[Corollary 2.6]{CD}. As
a sanity check we note  
that
\[
  \frac d {2g-2} = 1 + \frac 1 {k-1},
\]
so that indeed the degree does not exceed $4g+4$.

These examples are Veronese re-embeddings of cones as studied in
Section~\ref{S:supraCSegre} below. After we prove the claim, we
briefly discuss the variant when we consider
\[
  H_k^{(p')} = k \bigl(C_0+F(p')\bigr)
\]
with $p'\neq p$.

\begin{proof}[Proof of the claim]
One has 
$d=H_k^2 = k^2$
and $H_k \cdot C_0=0$.
Moreover,
\(
-K_R \lineq 2C_0 + F(p)
\),
hence
\[
  \restr [K_R+H_k] {C_0}
  = \restr [H_{k-2}+F(p)] {C_0}
  = p,
\]
and thus the adjoint system $|K_R+H_k|$ has a base point at $p \in
C_0$.
A direct computation gives
\begin{align*}
  (K_R+H_k)\cdot H_k
    = k(k-1),
\end{align*}
and therefore the members of $|H_k|$ have genus
$g= \frac 1 2 k(k-1)+1$.
Then it is straightforward that $d = 2g-2+k$.

One has
\begin{align*}
  h^0(R,H_k)
  &= h^0(\Gamma, \Sym^k (\O_\Gamma\oplus \O_\Gamma(-p))
    \otimes \O_\Gamma(kp))
  \\
  &= h^0(\Gamma, \O_\Gamma(kp) \oplus \O_\Gamma((k-1)p)\oplus \cdots
    \oplus \O_\Gamma(p)\oplus \O_\Gamma)
  \\
  &= k+\cdots+2+1+1
  \\
  &= 1 + \frac 1 2 k(k+1),
\end{align*}
and similarly
\begin{align*}
  h^1(R,H_k)
  &= h^1(\Gamma, \O_\Gamma(kp) \oplus \O_\Gamma((k-1)p)\oplus \cdots
    \oplus \O_\Gamma(p)\oplus \O_\Gamma)
  \\
  &= h^1(\O_\Gamma) = 1.
\end{align*}
This, by Proposition~\ref{qd}, \eqref{qd-v},
implies that $S$ has linearly normal hyperplane sections.
\end{proof}

If, on the other hand, one considers
$H_k^{(p')} = k \bigl(C_0+F(p')\bigr)$
for $p'\in C_0$ such that
$kp' \not\lineq kp$,
then
\[
  \restr [K_R+\smash{H_k^{(p')}}] {C_0}
  = \restr [(k-2)C_0+F(kp'-p)] {C_0}
  \lineq p'+(k-1)(p'-p),
\]
and thus the adjoint system $|K_R+H_k|$ has a base point at
the unique point $p^{(k)} \in C_0$ linearly equivalent to
$p'+(k-1)(p'-p)$.
The numerical characters computed above remain unchanged, but
\begin{align*}
  h^0(R,H_k^{(p')})
  &= h^0(\Gamma, \Sym^k (\O_\Gamma\oplus \O_\Gamma(-p))
    \otimes \O_\Gamma(kp'))
  \\
  &= h^0(\Gamma, \O(kp') \oplus \O((kp'-p)\oplus \cdots \oplus
    \O(kp'-(k-1)p)\oplus \O(kp'-kp))
  \\
  &= k+\cdots+2+1+0
  \\
  &= \frac 1 2 k(k+1),
\end{align*}
and 
\begin{align*}
  h^1(R,H_k^{(p')})
  &= h^1(\Gamma, \O(kp') \oplus \O((kp'-p)\oplus \cdots \oplus
    \O(kp'-(k-1)p)\oplus \O(kp'-kp))
  \\
  &= h^1(\O(p'-p)) = 0.
\end{align*}
This implies that for all $C \in |H_k^{(p')}|$,
the linear system $|H_k^{(p')}|$ cuts out a codimension one linear
subsystem of $|\restr {\smash{H_k^{(p')}}} C|$
on $C$.
For $k\geq 3$, there are then two possibilities: either this subsystem
is base-point-free, and then the surface gotten from $|H_k^{(p')}|$
has degree $k^2$ and hyperplane sections that are not linearly normal,
or this subsystem has a unique base point, and then the surface gotten
from $|H_k^{(p')}|$ has degree $k^2-1$ and linearly normal hyperplane
sections.
\end{expl}

\begin{expl}
\label{ex:ell_ruled_det0-dec}
Let $\fe$ be a non-torsion degree $0$ divisor
on $\Gamma$, and $\mathcal{E}=\O_\Gamma\oplus \O_\Gamma(\fe)$.
We consider $R=\P(\mathcal{E})$, with the notation introduced above,
at the beginning of Section~\ref{s:examples}.
The curve $C_0$ is the unique effective divisor in its linear
equivalence class, and there is a unique, irreducible, curve
$C_\fe \lineq C_0-F(\fe)$. One has
$-K_R \lineq C_0 + C_\fe$.
Let $p$ be a general point on $\Gamma$.
For all $g\in \N$, let
$H_g \lineq gC_0+F(p)$.

The following properties hold.
For all $g\geq 3$, the linear system $|H_g|$ has dimension $g$; it has
two base points on $C_0$ and $C_\fe$ respectively, and its proper
transform $|H'_g|$ on the blow-up 
$R'\to R$ at these two base points defines a birational morphism
$R' \to S \subset \P^g$ onto a surface with canonical hyperplane
sections. 
The proper transforms $C'_0$ and $C'_\fe$ of $C_0$ and $C_\fe$ have
self-intersection $-1$ and intersect $H_g'$ trivially;
in particular, the surface $S \subset \P^g$ has two elliptic double
points. 
The adjoint system $|K_{R'}+H'_g|$ has two base points, lying on
$C'_0$ and $C'_\fe$ respectively.

We omit the proof of these properties, because it is similar to the
proof of the analogous properties in the next example, which is
thoroughly treated in \cite{edoardo-TVcurves}.
\end{expl}

The next example is kind of a degenerate version of the previous one:
the anticanonical divisor becomes a section with multiplicity two, and
the two base points of the linear system $|H_g|$ become infinitely
near. The adjoint system $|K_{R'}+H_g'|$ has a fixed part, and the
surface image of $|H_g'|$ has an interesting singularity.

\begin{expl}
\label{ex:ell_ruled_det0-indec}
Let $\mathcal{E}$ be the unique indecomposable vector bundle
with trivial determinant on $\Gamma$,
see \cite[Thm.~2.15, p.~377]{hartshorne}, and consider $R =
\P(\mathcal{E})$ with the notation introduced at the beginning of
Section~\ref{s:examples}. 
The curve $C_0$ is the unique effective divisor in its linear
equivalence class, and one has
$-K_R \lineq 2C_0$.
Let $p$ be a general point on $\Gamma$.
For all $g\in \N$, let
$H_g \lineq gC_0+F(p)$.

The following properties hold, see \cite{edoardo-TVcurves}.
For all $g\geq 3$, the linear system $|H_g|$ has dimension $g$; it has
two infinitely near base points, with the proper base point lying on
$C_0$. The proper
transform $|H'_g|$ on the blow-up 
$R'\to R$ at these two base points defines a birational morphism
$R' \to S \subset \P^g$ onto a surface with canonical hyperplane sections.
The proper transform $C'_0$ of $C_0$ has
self-intersection $-1$ and intersects $H_g'$ trivially.
The surface $S \subset \P^g$ has a genus two singularity,\footnote
{This means that, denoting by $f:S'\to S$ of the singularity, the
  stalk of $R^1f_* \O_{S'}$ at the singular point has dimension $2$,
see \cite[Chap.~I, Definition~5.1]{epema}.}
consisting of a double point with an infinitely near double line.
The adjoint system $|K_{R'}+H'_g|$ has the exceptional locus of $R'
\to R$ as a fixed part.
\end{expl}

\endgroup

\subsection{Further properties of the adjoint system}
\label{s:adj-further}

\begin{prop}\label{lem:due}
Let $S, S'$ and $H$ be as in Section \ref {ssec:setup}.  Assume that
$h^0(S', K_{S'}+H)>0$ and $(K_{S'}+H)^2>0$.
Let $D$ be a divisor in  $|K_{S'}+H|$. Then: 
\begin{compactenum}[(i)]
\item\label{due-i}  $h^1(S', K_{S'}+H)=h^1(S', 2K_{S'}+H)=0$; 
\item\label{due-ii}  $h^0(D, K_{S'}+H)=g-1$ and $h^1(D ,K_{S'}+H)=0$; 
\item\label{due-iii}  $h^0(S', 2K_{S'}+H)=h^0(D, \omega_D) -q$;
\item\label{due-iv}
  $p_a(K_{S'}+H) \geq q$,
  and equality holds if and only if $h^0(S', 2K_{S'}+H)=0$;
\item\label{due-v} $p_a(K_{S'}+H)-1= K_{S'}^2+3g-3-d$;
\item\label{due-vi} $(K_{S'}+H)^2= g-2+p_a(K_{S'}+H)\geq g-2+q$.
\end{compactenum} 
\end{prop}

\begin{proof}  
\eqref{due-i} comes from the Kawamata-Viehweg theorem,
because both $H$ and $K_{S'}+H$ are nef and big,
see Proposition \ref{nef}, (i).

For \eqref{due-ii}, consider the long exact sequence coming from
\[
  0\to
\mathcal O_{S'}\to  \mathcal O_{S'}(K_{S'}+H)\to \mathcal
O_{D}(K_{S'}+H)\to 0,
\]
namely
\begin{multline*}0\to H^0(S',\mathcal O_{S'})\to H^0(S',K_{S'}+H)\to
  H^0(D, K_{S'}+H) 
  \\\to H^1(S', \mathcal O_{S'})\to
  H^1(S',K_{S'}+H)\to H^1(D,K_{S'}+H) \to H^2(S', \mathcal O_{S'}) = 0.
\end{multline*}
Then, \eqref{due-ii} follows from 
$h^0(S',K_{S'}+H)=g-q$ (Proposition \ref{qd}, \eqref{qd-lem:adj}),
$h^1(S', K_{S'}+H)=0$, and
$ h^2(S',\mathcal O_{S'})=h^0(S', K_{S'})=0$.

\eqref{due-iii}
follows from the long exact sequence  coming from $$0\to \mathcal O_{S'}(K_{S'})\to  \mathcal O_{S'}(2K_{S'}+H)\to \omega_D\to 0$$ and  the fact that $h^1(S', 2K_{S'}+H)=0$. 

For \eqref{due-iv}, note  that $K_{S'}+H$ is 1-connected because it is
nef and big (Lemma \ref{l:1-conn}), therefore
$p_a(K_{S'}+H)= h^0(D,\omega_D)$. Then, \eqref{due-iv} follows from
\eqref{due-iii}.

For \eqref{due-v}, just use the adjunction formula.  

For \eqref{due-vi}, note that by the Riemann--Roch theorem  and
assertion \eqref{due-ii},
$ g-1=h^0(D,K_{S'}+H)=(K_{S'}+H)^2+1-p_a(K_{S'}+H) $, 
and so $(K_{S'}+H)^2=g-2+p_a(K_{S'}+H)$.
The inequality follows from \eqref{due-iv}.
\end{proof}

\begin{prop}\label{lbirational}
Let $S, S'$ and $H$ be as in Section \ref {ssec:setup},
and assume that
$h^0(S', K_{S'}+H)>0$ and $(K_{S'}+H)^2>0$.  If
$H^2 \geq 10$,
then the map defined by
$|K_{S'}+H|$ is birational.
\end{prop}

\begin{proof}
Suppose by contradiction that the map defined by
$|K_{S'}+H|$ is not birational.
Let $x\in S'$ be a general point, and $y\in S'$ be such that $x\neq y$
and $x$ and $y$ have the same image. Then, by the main theorem 2.1 of
\cite{BS}, there is an effective divisor $D$, containing $x$ and $y$,
such that
\[
  D\cdot H -2 \leq D^2 \leq 1.
\]
Since $h^0(S', K_{S'}+H)>0$, $S$ is not a scroll, hence $D\cdot H >1$, and
$D\cdot H$ equals $2$ or $3$.
If $D\cdot H =3$, then $D^2=1$, and this contradicts the index theorem
because $H^2 \geq 10$. So we must have $D \cdot H = 2$, and $S$ is
swept out by irreducible conics. Let $A$ be the preimage in $S'$ of a
general member of this family of conics.  We claim that $A^2=0$. Indeed, since $H\cdot A=2$ and $H^2\geq 10$, if $A^2\geq 1$, then the Hodge index theorem gives
$$
4=(H\cdot A)^2\geq H^2 A^2\geq H^2\geq 10,
$$
a contradiction. Thus,  $A$ moves in a (possibly
non-linear) pencil. Now, $H\cdot A=2$ and $K_{S'} \cdot
A = -2$, hence $A\cdot (K_{S'}+H) = 0$, and therefore every member of
$|K_{S'}+H|$ consists of curves contained in members of the pencil
$\{A\}$; this implies that $(K_{S'}+H)^2 \leq 0$, in contradiction
with our assumptions. This ends the proof.
\end{proof}

\section 
{The irrational case}
\label{S:irreg}

Here we consider surfaces with $q>0$ and $h^0(S', K_{S'}+H)>0$; the
latter assumption is equivalent to $g>q$ by Proposition~\ref{qd}, \eqref{qd-lem:adj}.
We set $\mu=H\cdot G$, where $G$ denotes the general fibre of the
Albanese pencil. 
Since $g>q$, we have $\mu\geq 2$;
on the other hand, $g-1\geq  \mu(q-1)$ by Riemann--Hurwitz.

\subsection{The case $(K_{S'}+H)^2= 0$}

\begin{lem}\label{lem:uno}
Let $S, S'$ and $ H$ be as in Section \ref {ssec:setup},
and assume $q>0$ and $h^0(S', K_{S'}+H)>0$.

If $\mu=2$ (hence $S$ is ruled by conics), then $(K_{S'}+H)^2= 0$ (and so $K_{S'}^2+4g-4-d=0$), and  $|K_{S'}+H|$  is composed with the Albanese pencil.

Conversely, if $(K_{S'}+H)^2=0$, then $|K_{S'}+H|$  is composed with the Albanese pencil  and $\mu =2$. 
\end{lem}

\begin{proof}
The first assertion follows from Lemma \ref{grau}.
The converse is  Proposition \ref{feixe}.
\end{proof}


\begin{lem}\label{prop:irr}
Let $S, S'$ and $ H$ be as in Section \ref {ssec:setup},
and assume $q>0$ and $h^0(S', K_{S'}+H)>0$.
If $d>3g-3$, 
then $\mu=2$ and $g> 8q-7$.
\end{lem}

\begin{proof}
By \cite [Thm. (2.3)]{Ha}, one has
$$
d\leq \frac {2\mu}{\mu-1}(g-1).
$$
If $\mu\geq 3$, this  yields $d\leq 3(g-1)$, a contradiction. Hence  we have $\mu=2$. 

So, by Lemma \ref{lem:uno}, $ (K_{S'}+H)^2=0$,
i.e., $K_{S'}^2+4g-4-d=0$.
Then the hypothesis $d>3g-3$ gives $K_{S'}^2+g-1>0$,
and so $g>-K_{S'}^2+1\geq 8q-7$.
\end{proof}

\subsection{The case $(K_{S'}+H)^2> 0$}

\begin{lem}\label{bounds}
  Let $S, S'$ and $ H$ be as in  Section
  \ref{ssec:setup}, and assume that  $q>0$, 
  $h^0(S', K_{S'}+H)>0$,  and $(K_{S'}+H)^2>0$.  Then $g\geq 9q-7$. 
  Furthermore, if equality holds, then
  $S'$ is minimal, $d=2g-1$, $p_a(K_{S'}+H)=q$,
  and $h^0(S', 2K_{S'}+H) =0$.
\end{lem}

\begin{proof}
From $q>0$, we have $K_{S'}^2\leq 8-8q$. Moreover, we assume $d\geq
2g-1$. Then:
\begin{align*}
  (K_{S'}+H)^2
  &= K_{S'}^2+4g-4-d
  \\
  & \leq 8-8q+4g-4+1-2g=5-8q+2g.
\end{align*}
On the other hand, by Proposition \ref{lem:due}, \eqref{due-vi},
$(K_{S'}+H)^2 \geq g-2+q$.
Thus,
\[
  5-8q+2g\geq g-2+q,
\]
i.e., $g\geq 9q-7$. 
If equality holds, then all the above inequalities are equalities, in
particular $K_{S'}^2 = 8-8q$, $d=2g-1$, and equality holds in 
Proposition \ref{lem:due}, \eqref{due-vi}, which means
$p_a(K_{S'}+H)=q$, and then
$h^0(S', 2K_{S'}+H) =0$ by 
Proposition \ref{lem:due}, \eqref{due-iv}.
\end{proof}

\begin{prop}\label{nopencil} Let $S, S'$ and $ H$ be as in  Section \ref
  {ssec:setup},
  and assume that  $q>0$, $h^0(S', K_{S'}+H)>0$,  and
  $(K_{S'}+H)^2>0$.
  If $g-q\geq 3$, then $|K_{S'}+H|$ is not composed with a pencil.
\end{prop}

\begin{proof}  
Suppose by contradiction that $|K_{S'}+H|$ is composed with a pencil
$\{P\}$.  Since $h^0(K_{S'}+H)=g-q$,   we can  then write 
$$K_{S'}+H\equiv \alpha P+Z$$ where $Z$ is the fixed part (possibly
zero), and $\alpha\geq g-q-1$, with equality holding if and only if
the pencil is rational.
Being $K_{S'}+H$ nef and big, it is $1$-connected
(see Lemma~\ref{l:1-conn}), so
if $Z\neq 0$ then $P\cdot Z >0$.

Since  $h^0(S', K_{S'}+H)>0$ and $(K_{S'}+H)^2>0$, if $\{P\}$ is the
Albanese pencil, then $H\cdot P\geq 3$, by Lemma \ref{lem:uno}.
If $\{P\}$ is not the Albanese pencil, then  the general curve in
$\{P\}$ has genus at least $q$ and so, since the map defined by $|H|$ is
birational, we must have $H\cdot P\geq 3$ in this case as well.
The upshot is that $H\cdot P \geq 3$ in any event.
  
Since $H$ is nef and $H\cdot (K_{S'}+H)=2g-2$,
we obtain $\alpha H\cdot P\leq 2g-2$, and thus
\begin{equation}
\label{eq:nopencil}
  3 \alpha \leq \alpha H\cdot P\leq 2g-2.
\end{equation}
If $\alpha\geq g-q$, we obtain    $3g-3q\leq 2g-2$, i.e., $g\leq 3q-2$.  
Since, by Lemma \ref{bounds},  $g\geq 9q-7$, we see that
$\alpha\geq g-q$ is impossible. 

Therefore, one has $\alpha = g-q-1\geq 2$, and the pencil $\{P\}$ is
rational; in particular, $\{P\}$ is not the Albanese pencil, and
therefore $p_a(P)>0$.
From \eqref{eq:nopencil}, we obtain 
\[
  3\leq \frac {2g-2}  {g-q-1} =
  2+ \frac{2q} {g-q-1},
\]
hence $2q\geq g-q-1$, i.e., $g \leq 3q+1$.
Since,
by  Lemma \ref{bounds}, $g\geq 9q-7$, we obtain $6q \leq 8$, hence
$q=1$.
Then, since $g-q\geq 3$ and $g \leq 3q+1$,  we have $g=4$, and
$\alpha = 2$.

Suppose this case occurs.
By the index theorem, we have   $H^2(K_{S'}+H)^2\leq 4(g-1)^2=36$.
Since $H^2=d\geq 7$, this implies that $(K_{S'}+H)^2\leq 5$. 
Moreover, being composed with a pencil,
$|K_{S'}+H|$  has base points. So, by Proposition \ref{connected},
there is a curve $\theta$ such that $\theta^2=-1$, $\theta \cdot
K_{S'}=1$, and $\theta\cdot H=0$.    
So,
\[
  1= \theta \cdot (K_{S'}+H)=\theta \cdot (2P+Z)=2\theta\cdot P+
  \theta\cdot Z,
\]
which implies that $Z\neq 0$. 
Now notice that, since
\[
  4P^2+2P\cdot Z= 2P\cdot  (K_{S'}+H)
  \leq (2P+Z)\cdot (K_{S'}+H) \leq 5,
\]
one must have $P^2=0$.
But then $p_a(P)\geq 1$ implies that $P\cdot K_{S'}\geq 0$,
and so $2P\cdot (K_{S'}+H)\geq 2P\cdot H\geq 6$: this contradicts $
(K_{S'}+H)^2\leq 5$.  


We conclude that $|K_{S'}+H|$ is not composed with a pencil. 
\end{proof}

In fact, as we will now prove, the condition $g-q\geq 3$ always holds.

\begin{thm}\label{thm:npp} Let $S, S'$ and $ H$ be as in Section \ref {ssec:setup} and assume that $q>0$, $h^0(S', K_{S'}+H)>0$ and $(K_{S'}+H)^2>0$. Then $g-q\geq 3$, and therefore $|K_{S'}+H|$ is not composed with a pencil. 
\end{thm}

The proof starts by the following lemma, which shows that there is
only one numerical possibility.

\begin{lem}\label{nogq2} Let $S, S'$ and $ H$ be as in Section \ref
  {ssec:setup} and assume that $q>0$, $h^0(S', K_{S'}+H)>0$ and
  $(K_{S'}+H)^2>0$. Then $g-q\geq 3$ unless, possibly, if $q=1$,
  $g=3$, $d=5$, $h^0(S', H)=4$, $(K_{S'}+H)^2=2$, and $K_{S'}^2=-1$. 
 \end{lem}

\begin{proof}
From $h^0(S', K_{S'}+H)>0$, we know that $g>q$
by Proposition~\ref{qd}, \eqref{qd-lem:adj}.
   
Assume that $g-q\leq 2$.
Then, by Lemma \ref {bounds}, one has $g-q\geq 8q-7$, and so we get $q=1$, implying that either $g=2$ or $g=3$. 

Let $g=2$. We have $0<(K_{S'}+H)^2=K_{S'}^2+4g-4-d=K_{S'}^2+4-d$, 
see Proposition \ref{qd}, \eqref{qd-viii}.
Since $q=1$ implies $K_{S'}^2\leq 0$, we have $4-d>0$, and since $d\geq 2g-1=3$,  we conclude that $d=3$.  But if $d=3$, the map defined by $|H|$ cannot be birational because curves of degree 3 have genus at most 1. So $g=2$ does not occur.

If $g=3$, by Proposition \ref {lem:due}, \eqref{due-vi}, we have
$(K_{S'}+H)^2\geq 2$, yielding, as above,
that $d$ can only be $5$ or $6$,
hence also $(K_{S'}+H)^2\leq 3$.  Since $h^0(S',
K_{S'}+H)=g-q=2$, $|K_{S'}+H|$ has base points and so there is a curve
$\theta$ as in   Proposition \ref {connected}.

Now, $(K_{S'}+H+\theta)^2= (K_{S'}+H)^2+1$. 
Moreover, $ 4= (K_{S'}+H)\cdot H= (K_{S'}+H+\theta)\cdot H$. Then, by the index theorem,  $H^2(K_{S'}+H+\theta)^2-16\leq 0$, and  we are left with the only  possibility  $d=5$, $(K_{S'}+H)^2=2$. In this case, from $2=(K_{S'}+H)^2=K_{S'}^2+4g-4-d=K_{S'}^2+3$, we get  $K_{S'}^2=-1$. Finally we have $h^0(S', H)=4$,  because curves of degree $5$  in $\P^n$ with $n>2$ have genus smaller than 3 by Castelnuovo's bound. \end{proof}

Therefore, it only remains to prove that the sole possibility left open
by Lemma \ref {nogq2} in fact does not occur.

\begin{proof}[Proof of Theorem~\ref{thm:npp}]
By Proposition \ref{nopencil} and Lemma \ref{nogq2}, it is enough to show
that  there is no surface  $S$ with $q=1$,  $g=3$, $d=5$, $h^0(S',
H)=4$, $(K_{S'}+H)^2=2$, and $K_{S'}^ 2=-1$. 
So let us assume from now on, by contradiction, that 
$q=1$,  $g=3$, $d=5$, $h^0(S', H)=4$, $(K_{S'}+H)^2=2$, $K_{S'}^2=-1$. As we have seen in the proof of Proposition \ref {nogq2}, $|K_{S'}+H|$ has base points and so there is an irreducible curve $\theta$ of genus 1 with $\theta^2=-1$ and $H\cdot \theta=0$, as in   Proposition \ref {connected}.

Consider the linear system  $|K_{S'} + H+\theta|$.  Look at the exact sequence
$$
0\to \mathcal O_{S'}(K_{S'} + H)\longrightarrow  \mathcal O_{S'}(K_{S'}+H+\theta)  \longrightarrow \mathcal O_{\theta}(K_{S'}+H+\theta) \longrightarrow 0.
$$
We have $h^0(S', K_{S'} + H)=g-q=2$ and $h^1(S', K_{S'} + H)=0$. Moreover,
$\restr[K_{S'} + H+\theta] \theta$ is trivial on  $\theta$
(it is $K_\theta+\restr H \theta$),  hence 
$h^0(\theta, \mathcal O_{\theta}(K_{S'}+H+\theta) )=1$. 
So we have $h^0(S',K_{S'} + H+\theta)=3$, and $\theta$ is not a component of the fixed part of  $|K_{S'} + H+\theta|$ (if any). 
  Note that $(K_{S'} + H+\theta)^2=3$ and $K_{S'} \cdot (K_{S'} + H+\theta)=-1$.

We are going to see  that the general curve in $|K_{S'} + H+\theta|$ is irreducible. \medskip

Claim 1)  {\sl $|K_{S'} + H+\theta|$  is not composed with a pencil.} 
\medskip

Suppose otherwise. Then we can write $|K_{S'} + H+\theta|=\alpha P+Z$ where  $\{P\}$ is a pencil, $\alpha\geq 2$ (because $h^0(S',K_{S'} + H+\theta)=3$) and $Z$ (possibly zero) is the fixed part. 

 From  $(K_{S'} + H+\theta) \cdot H=4$, we have $H\cdot P\leq 2$.  Since  $(K_{S'}+H)^2>0$, $\mu\geq 3$ by Lemma \ref {lem:uno},  and so  $\{P\}$ is not the Albanese pencil.   
 
 But then the general curve $P \in \{P\}$ has arithmetic genus $\geq 1$,  and so  $H\cdot P\leq 2$ gives a contradiction to the fact that 
 the map defined by $|H|$ is birational.

So $|K_{S'} + H+\theta|$  is not composed with a pencil and the claim is proven.

\medskip

Claim 2)  {\sl  The general curve in $|K_{S'} + H+\theta|$  is irreducible.} 
\medskip

Suppose otherwise. Then there is a fixed divisorial part $Z$ and we can 
write  $|K_{S'} + H+\theta|=|M|+Z$  where $M$ is the moving part.  Since, by Claim 1), $|K_{S'} + H+\theta|$ is not composed with a pencil, the general curve in $|M|$ is irreducible.  Note that, since $S'$ is not birational to $\mathbb P^2$, one has $M^2\geq 2$,  since $\dim (|M|)=\dim (|K_{S'} + H+\theta|)=2$.

Remark now that, since $K_{S'} + H$ is nef and $(K_{S'} +
H+\theta)\cdot \theta =0$, also $K_{S'} + H+\theta$ is nef.
Moreover, $(K_{S'} + H+\theta)^2 =3$, hence $K_{S'} +
H+\theta = M+Z$ is also big, and therefore $1$-connected by Lemma
\ref{l:1-conn}.
Thus, $M\cdot Z \geq 1$, and $(M+Z) \cdot Z \geq 0$.
Since $(M+Z)^2=(K_{S'} + H+\theta)^2=3$, 
the only possibility is $M^2=2, M\cdot Z=1, Z^2=-1$, and $(K_{S'} +
H+\theta)\cdot Z=0$.  

Since, as we saw above, $\theta$ is not a component of $Z$
(the restriction map $H^0(S',K_{S'}+H+\theta) \to H^0(\theta,
\O_\theta)$ is surjective),
one has $\theta\cdot Z\geq 0$. Then from $H\cdot Z\geq 0$ and $(K_{S'} + H+\theta)\cdot Z=0$,  we conclude that  $K_{S'}\cdot Z\leq 0$. 
From  $M\cdot Z=1$ we know that $Z$ is 1-connected hence, from the adjunction formula, we obtain $K_{S'}\cdot Z=-1$. So $Z$ is a $(-1)$--divisor, and in fact it is a single $(-1)$-curve because $K_{S'}^2=-1$.   Hence, because $|H|$ does not contract $(-1)$-curves, one has $H\cdot Z> 0$. Therefore, from $H\cdot (M+Z)=H\cdot (K_{S'} + H+\theta)=4$, we have $H\cdot M\leq 3$. 

On the other hand, from $K_{S'} \cdot (M+Z)=K_{S'} \cdot (K_{S'} +
H+\theta)=-1$ and $K_{S'} \cdot Z=-1$, we obtain $K_{S'}
\cdot M=0$, implying that $p_a(M)=2$. This is
in  contradiction with $H\cdot M\leq 3$
because $|H|$ defines a birational map. Hence $Z=0$, and Claim 2) is proven.

\medskip 

From Claims 1) and 2),  the general curve $M$ in 
$|K_{S'} + H+\theta|$ is irreducible;
moreover, as we have seen, $M^2=3$
and $K_{S'}\cdot M=-1$, hence $p_a(M)=2$.

Since $H\cdot M= H\cdot (K_{S'} + H+\theta)=4$,
the  images in $S$ of the
curves in $|K_{S'} + H+\theta|$ must be plane quartics of genus 2,
and the residual with respect to $|H|$ are lines. So 
$S$ would contain infinitely many lines, which is not
possible.
\end{proof}

Examples \ref{ex:ell_ruled_det0-dec} and \ref{ex:ell_ruled_det0-indec} with $g=3$ show that there exist surfaces $S, S'$ with $H$ and $H$ as in Section \ref {ssec:setup} except that $d=2g-2$, with  $q=1$,  $g=3$, $d=4$, $h^0(S', H)=4$, $(K_{S'}+H)^2=2$, $K_{S'}^2=-1$.

\subsection{Empty biadjoint system}
\label{s:empty-2adj}

In this section, we consider the case when $q>0$,
and $h^0(S', 2K_{S'}+H)=0$.
Observe that if $(K_{S'}+H)^2=0$ then, by Lemma~\ref{lem:uno},
$|K_{S'}+H|$ is composed with the Albanese pencil
and thus $h^0(S', 2K_{S'}+H)=0$.
We will concentrate on the case when $(K_{S'}+H)^2 > 0$.

\begin{prop}
\label{p:empty-2adj}
Let $S,S',H$ be as in Section~\ref{ssec:setup},
and assume that $q>0$, 
$h^0(S',K_{S'}+H) >0$,
$h^0(S',2K_{S'}+H)=0$, and
$(K_{S'}+H)^2 \geq 5$.
Then $\mu = 3$.
Moreover, any curve in $|K_{S'}+H|$ consists of a smooth, irreducible,
curve of genus $q$, plus possibly curves contained in the fibres of
the Albanese map.
\end{prop}

\begin{proof}
First note that $K_{S'}+H$ is nef by Proposition~\ref{nef}, (i),
hence also big since we assume $(K_{S'}+H)^2 \geq 5$.
By the main theorem 2.1 of \cite{BS},
since $h^0(S', 2K_{S'}+H)=0$ and $(K_{S'}+H)^2 \geq 5$, 
if $x$ is a general point of $S'$, there is an effective divisor $D$
containing $x$ such that
\[
  -1 \leq (K_{S'}+H)\cdot D -1 \leq D^2 \leq 0.
\]
---
If $D^2 = -1$, then
$(K_{S'}+H)\cdot D = 0$.
We claim that $D$ is $1$-connected:
let $A$ and $B$ be non-zero, effective divisors, such that $D=A+B$.
Since  $K_{S'}+H$ is nef, we have 
$(K_{S'}+H)\cdot A = (K_{S'}+H)\cdot B = 0$ hence,
by the index theorem, $A^2<0$ and $B^2<0$. Then $(A+B)^2=-1$ implies
$A \cdot B >0$, which proves the claim.
Now, since $H\cdot D \geq 0$, one has $K_{S'} \cdot D \leq 0$,
and thus $D$ is a $(-1)$-divisor. Since $D$ contains a general point
of $S'$, this is impossible.
\\---
If $D^2 = 0$, then 
$(K_{S'}+H)\cdot D \leq 1$;
In fact, $(K_{S'}+H)\cdot D =1$, since $(K_{S'}+H)\cdot D >0$ by the
index theorem.
Besides, since $D$ passes through a general point, one has
$H\cdot D \geq 1$, hence $K_{S'} \cdot D \leq 0$.
If $K_{S'} \cdot D = 0$, then $H\cdot D = 1$.
Let $D_0$ be an irreducible component of $D$ that is movable.
Then $H\cdot D_0 \geq 1$ and, since $H$ is nef, in fact
$H\cdot D_0 = 1$, and thus $S$ is a scroll, in contradiction with 
$h^0(K_{S'}+H) >0$.
Otherwise, $K_{S'}\cdot D <0$. Since $D$ is $1$-connected by
Lemma~\ref{reider-1con}, $K_{S'}\cdot D =-2$, and $D$ is rational,
hence it moves in the Albanese pencil $\{G\}$.
Then, $H\cdot D = 1 -K_{S'}\cdot D = 3$ implies that
$\mu = H\cdot G = 3$, as we wanted to show. 

Let $A$ be a member of $|K_{S'}+H|$.
One has $A\cdot G = H\cdot G + K\cdot G =1$.
Since $G$ is nef, there is one irreducible component $A_0$ of $A$ such
that $A_0 \cdot G = 1$, and thus $A_0$ is smooth, irreducible, of
genus $q$, whereas $A_i\cdot G=0$ for any other component of $A$.
This proves the assertion.
\end{proof}

\begin{rem}\label{rem:otto}  In the situation of Proposition \ref {p:empty-2adj}, one can have $(K_{S'}+H)^2<5$ only if $q=1$,
and either $g=5$ and $(K_{S'}+H)^2=4$, or $g=4$ and $(K_{S'}+H)^2=3$. 

Indeed, assume $(K_{S'}+H)^2>0, q>0$, $h^0(S', K_{S'}+H)>0$ and $h^0(S', 2K_{S'}+H)=0$ and $(K_{S'}+H)^2<5$.  
By Proposition \ref {lem:due}, \eqref {due-iv} and \eqref  {due-vi}, we have $(K_{S'}+H)^2=g-2+q$ and  by Lemma \ref {bounds}, $g\geq 9q-7$, i.e., $g+q\geq 10q-7$.   Now $(K_{S'}+H)^2<5$  yields $g+q<7$, so  $(K_{S'}+H)^2<5$ can only occur if $q=1$.
Since $g-q\geq 3$ and $g<7-q=6$ we have only the possibilities $g=5$, $(K_{S'}+H)^2=4$ or $g=4$, $(K_{S'}+H)^2=3$.
\end{rem}

\noindent
We shall use the following lemma to give an application to
Proposition~\ref{p:empty-2adj} above.

\begin{lem}
\label{l;empty-2adj}
Let $S,S',H$ be as in Section~\ref{ssec:setup},
and assume that $q>0$, and $\mu=3$.
If $H^2\geq 10$ and $h^0(S', K_{S'}+H)>0$, then $|K_{S'}+H|$ has no base points.
\end{lem}

\begin{proof}
We prove the contrapositive.
By Proposition~\ref{prop:bpf},
if $K+H$ has base points, then
there is a genus $q$ curve $\theta$ such that $\theta^2=-1$, $\theta \cdot
H=0$.
Let $\{G\}$ be the Albanese pencil, and let
$\alpha = \theta\cdot G$.
Since $\theta$ has genus $q$,
one has $\alpha \geq 1$. 
Then the divisor $B:=G+\theta$ satisfies $B^2=2\alpha-1\geq 1$ and
$H\cdot B=\mu $. By the index theorem,
one has $H^2B^2\leq (H\cdot B)^2=\mu^2$,
hence $H^2 \leq 9$.
\end{proof}

\begin{cor}
Let $S,S',H$ be as in Section~\ref{ssec:setup},
and assume that $q>0$,
$h^0(S',K_{S'}+H)>0$, and
$h^0(S',2K_{S'}+H)=0$.
If $(K_{S'}+H)^2>0$ and  $H^2\geq 10$, 
then  $|K_{S'}+H|$ is base point free and determines a birational map, and  the image of $S'$ by the adjoint series $|K_{S'}+H|$
has degree $g-2+q$ in $\P^{g-q-1}$.
\end{cor}

Let us point out that $g-2+q$
is the maximal possible degree for a
surface in $\P^{g-q-1}$ with sectional genus $q$ equal to the
irregularity, see Proposition~\ref{qd}, \eqref{qd-iii}.

\begin{proof}
By Proposition~\ref{lbirational},
the linear system $|K_{S'}+H|$ defines a birational map.  Suppose first that $(K_{S'}+H)^2< 5$. Then, by Remark \ref {rem:otto}, the only possibilities are
$$
\text{either}\  (q,g,(K_{S'}+H)^2) =(1,4,3),
\quad \text{or} \ (q,g,(K_{S'}+H)^2) =(1,5,4).
$$

The former case is impossible. Indeed, since $q=1$, we have $K_{S'}^2\leq 0$. Then from the identity
$$
(K_{S'}+H)^2=K_{S'}^2 +4g-4-H^2
$$
we have $H^2=K_{S'}^2 +9\leq 9$, that is a contradiction. 

In the latter case one has $\dim (|K_{S'}+H|)=3$, and the map defined by $|K_{S'}+H|$  is birational onto its image $\Sigma$ that has degree $s\leq 4$, being $s<4$ only if $|K_{S'}+H|$ has base points. On the other hand, since $\Sigma$ is birational to $S'$, hence has irregularity 1, one has $s\geq 3$. If $s=4$, then the statement of the lemma holds. If $s=3$, then $\Sigma$ must be a cubic cone, and then one has $\mu=3$. Then 
Lemma \ref {l;empty-2adj} implies that $|K_{S'}+H|$ is base point free, and we have a contradiction. 

Finally we may assume $(K_{S'}+H)^2\geq 5$. Then   
by Proposition~\ref{p:empty-2adj}
and Lemma~\ref{l;empty-2adj}, $|K_{S'}+H|$
is base-point-free. 
Therefore, the image of $S'$ by this map is a degree $(K_{S'}+H)^2$
surface in $\P^{g-q-1}$ (see Proposition~\ref{qd}, \eqref{qd-lem:adj}).
By Proposition~\ref{lem:due}, \eqref{due-iv},
it has sectional genus $q$,
and by Proposition~\ref{lem:due}, \eqref{due-vi}, it has degree
$g-2+q$.
\end{proof}

\section{An extension of a theorem of C. Segre}
\label{S:supraCSegre}

In this section we continue our study of irrational surfaces as in
Section~\ref{ssec:setup} (in fact we consider slightly more
restrictive hypotheses).
The main result of this section is the following generalization of
C.~Segre's classical theorem  \cite[Thm. 2.3]{CD} mentioned in the
introduction. We shall apply it in particular to get
Corollary~\ref{cor:irr}.

\begin{thm}\label{thm:segre}
Let $C \subset \P^r$ be a (smooth) linearly normal, non--degenerate, projective curve of
genus $g$ and degree $d\geq 2g+5$ (so $r=d-g$).
If $C$ is a hyperplane section of an irregular surface $S \subset \P^{n}$, with $n=r+1$, ruled
in conics, then $S$ is either the 2-Veronese re-embedding of a cone or a simple internal projection thereof.
\end{thm}

Let $S \subset \P^n$ be a surface having $C$ as a hyperplane section
as in the theorem. It satisfies the assumptions of Section
\ref{ssec:setup}, and we will use the notation introduced there.
In particular, we let $\pi:S'\to S$ be the minimal desingularization
of $S$.
By abuse of notation we will denote by $C$ the proper transform of $C$ on
$S'$, that is isomorphic to $C$.
By assumption, the fibres of the Albanese map $S'\to \Gamma$
are mapped by $\pi$ to conics sweeping out $S$.  

We let $\gamma:S''\to \Gamma$ be a relative minimal model
of $S'$ so that there is a birational morphism $h: S'\longrightarrow
S''$ such that $\alb=\gamma\circ h$. Let $\bar C$ be the image of
$C$ via $h$. 

\begin{lem} \label{lem:cdr} The curve $\bar C$ is smooth. \end{lem}

\begin{proof} To prove this it suffices to show that $C$, which intersects positively any $(-1)$--curve, intersects any $(-1)$--curve in exactly one point. Let $E$ be a $(-1)$--curve. It is contained in a fibre of $\alb$. On the other hand any reducible fibre of $\alb$ either contains only one $(-1)$--curve with multiplicity 2 met by $C$ in one point, or it contains exactly two distinct $(-1)$--curves that are met by $C$ in one point (remember that  the intersection number of $C$ with the fibres of $\alb$ is 2). Indeed, since $C$ intersects any $(-1)$--curve positively, and since the intersection number of $C$ with the fibres of $\alb$ is 2, any such  fibre cannot contain more than two distinct $(-1)$--curves, and if it contains a $(-1)$--curve with multiplicity, it contains only that $(-1)$--curve with multiplicity 2. To finish the proof it suffices to show that a fibre of $\alb$ cannot contain only one $(-1)$--curve with multiplicity 1. In fact if $E$ is such a $(-1)$--curve, then its residual with respect to the fibre of $\alb$ in which it sits is a $(-1)$--divisor by Zariski's lemma, and therefore it must contain another $(-1)$--curve.\end{proof}


Next, $S$ is the image of $S''$ via a linear subsystem of
$|\bar C|$ which may have some simple base points. In that case,
we can replace $S$ with the image of $S''$ via $\varphi_{|\bar C|}$,
of which $S$ will be an internal projection. Note that the sectional
genus $g$ is not affected by this operation, whereas the degree may
increase, but the hyperplane sections stay linearly normal. So from
now on we may and will assume that $S$ is the image of $S''$ via the
complete linear system $|\bar C|$.

Write $S'' = \P(\sE)$, where $\sE$ is a normalized rank $2$ vector
bundle of degree $-e$. There exists an invertible sheaf $\sL$ on
$\Gamma$ and an exact sequence
\[
  0 \to \O \to \sE \to \sL \to 0
\]
with $\deg(\sL)=\deg(\sE)=-e$
(see \cite [p. 372, proof of Prop. 2.8]{hartshorne}).
Since $C$ is a hyperplane section of 
$S$, and the latter is swept out by conics, there exists an invertible
sheaf $\sA$ on $\Gamma$ such that 
\[
  \bar C \sim 2E+\gamma^* (A),
\]
where $E \in |\O_{\P(\sE)}(1)|$ and $A \in |\sA|$.
\def\fa{\alpha}
Let $\fa$ be the degree of $\sA$.

\begin{lem}
\label{l:cones_relations}
The following two relations hold,
\begin{align*}
  g &= 2q-1+\fa-e \\
  d &=  4g+4-8q.
\end{align*}
\end{lem}

\begin{proof}
We have
\[
  K_{S''} \sim -2E + \gamma^*(K_\Gamma+\sL)
  \]
(see \cite [p. 373, Lemma 2.10]{hartshorne}),
hence
\[
  K_{S''} \equiv -2E+(2q-2-e)G
  \quad
  \text{and}
  \quad
  \bar C \equiv 2E+\fa G,
\]
where $G$ denotes the numerical equivalence class of the fibres of
$\gamma$.
Then one computes
\[
  2g-2=(K_{S''}+\bar C)\cdot \bar C 
  = 2(2q-2-e+\fa)
\]
and
\[
 d=  \bar C^2 
  = 4(\fa-e),
\]
as wanted. \end{proof}


\noindent
The following observation is the keystone of our proof of  Theorem
\ref {thm:segre}.

\begin{lem}
\label{l:AtL_non-spec}
The following relation holds:
\[
  d - 2(g-1)
  =
  2 \bigl[
  \fa-e - 2(q-1)
  \bigr].
\]
Thus,
\[
  d\geq 2g+5
  \quad \iff
  \quad
  \fa-e \geq 2q+2.
\]

\end{lem}

\begin{proof}
This can be proved by using  the two relations
stated in the previous lemma.
Yet we find it more satisfactory to observe that the quantity on the
left-hand-side is $ -K_{S''}\cdot \bar C$,
and then 
\begin{align*}
 d - 2(g-1)= 
  -K_{S''}\cdot \bar C
  &= \bigl(2E-(2q-2-e)F \bigr)
    \cdot \bigl( 2E+\fa F \bigr)
  \\
  &= 2(\fa-e) - 2(2q-2).
\end{align*}
\end{proof}

\begin{proof}[Proof of Theorem \ref {thm:segre}]
 Since $C \subset \P^r$ is
linearly normal, one has
\begin{align*}
  h^0(S'',\O_{S''}(\bar C))-1 
  &= d-g+1 \\
  &= 3(\fa-e)-2(q-1).
\end{align*}
For $i=0,1$, one has
\[
  H^i(S'',\O_{S''}(\bar C))
  \cong H^i(\Gamma, \gamma_* (\O_{S''}(\bar C))
  = H^i(\Gamma, (\Sym^2 (\sE)) \otimes \sA).
\]
Moreover, there exists a locally free sheaf $\sQ$ such that we have
two exact sequences
\[
  0 \to \O \to \Sym^2 (\sE) \to \sQ \to 0
  \qquad \text{and}
  \qquad
  0 \to \sL \to \sQ \to \sL^{\otimes 2} \to 0,
\]
hence also
\[
  0 \to \sA \to \Sym^2 (\sE) \otimes \sA \to \sQ \otimes \sA \to 0
  \qquad \text{and}
  \qquad
  0 \to \sL \otimes \sA \to \sQ \otimes \sA \to \sL^{\otimes 2}
  \otimes \sA \to 0.
\]
One has
\begin{align*}
  \deg(\sA) &= \fa \\
  \deg(\sA\otimes \sL) &= \fa-e \\
  \deg(\sA\otimes \sL^{\otimes 2}) &= \fa-2e.
\end{align*}
Since we are assuming $d\geq 2g+5$, we have $\fa-e \geq 2q+2$ by
Lemma~\ref{l:AtL_non-spec},
hence $\sA\otimes \sL$ is non-special.
Taking this into account, we have the two exact sequences
\[
  0 \to H^0(\Gamma, \sA) \to H^0(\Gamma, \Sym^2 (\sE) \otimes \sA) \to
  H^0(\Gamma, \sQ \otimes \sA) \to H^1(\Gamma, \sA)
\]
and
\[
  0 \to H^0(\Gamma, \sL \otimes \sA) \to H^0(\Gamma, \sQ \otimes \sA) \to
  H^0(\Gamma, \sL^{\otimes 2} \otimes \sA) \to 0,
\]
so that
\[
  h^0(S'',  \O_{S''}(\bar C)) = h^0 (\Gamma, \Sym^2 (\sE) \otimes \sA)
  \leq
  h^0(\Gamma, \sA)+
  h^0(\Gamma, \sA\otimes \sL)+
  h^0(\Gamma, \sA\otimes \sL^{\otimes 2}),
\]
with equality holding if $H^1(\Gamma, \sA)=0$.

Let us first assume that $e\geq 0$.
Then
$\fa=\deg(\sA)$ is larger than $\deg(\sA\otimes\sL)$,
hence $\sA$ is non-special as well, and thus
\[
  h^0(\Gamma, \sA)+
  h^0(\Gamma, \sA\otimes \sL)+
  h^0(\Gamma, \sA\otimes \sL^{\otimes 2})
  = 3(\fa-e)-3(q-1)+i,
\]
with
\[
  i = h^1(\Gamma, \sA\otimes \sL^{\otimes 2}).
\]
Therefore the condition that $C$ is linearly normal implies
$i \geq q$, hence $i=q$ and
$\sA\otimes \sL^{\otimes 2} = \O_\Gamma$, equivalently
$\sA =  \sL^{\otimes -2}$, and in particular $\fa=2e$.
Then the curve $\bar C$ is a member of the linear system
$|\O_{\P(\sE')}(2)|$,
where  $\O_{\P(\sE')}(1)$ is defined relatively to the vector bundle
$\sE' = \sE \otimes \sL^{-1}$ and, of course, $S''=\P(\sE')$.
One has
\[
  \deg(\sE') = -\deg (\sL) = e = \fa-e \geq 2q+2
\]
by Lemma~\ref{l:AtL_non-spec}.
Then, by \cite[Lemma~1]{CaCiFlMi}, $\sE'$ splits as $\sL^{-1} \oplus \O_\Gamma$ if
$h^1(S'', \O_{\P(\sE')}(1)) \geq q$.

To prove this inequality, we will relate
$H^1(S'', \O_{\P(\sE')}(1))$ to
$H^1(S'', \O_{\P(\sE')}(2))$.
Since $\bar C \in |\O_{\P(\sE')}(2)|$ is linearly normal and
$\restr {\O_{S''}(\bar C)} {\bar C}$ is non-special, the restriction exact sequence of
$\bar C \subset S''$ gives an isomorphism
\[
  H^1(S'', \O_{\P(\sE')}(2)) = H^1(S'', \O_{S''}(\bar C))
  \cong H^1(\O_{S''})=q.
\]

Now, consider a general member $D$ of
the linear system $|\O_{\P(\sE')}(1)|$; 
since the vector bundle $\sE'$ is positive
enough, $D$ is a smooth curve
isomorphic to $\Gamma$, and $\restr {\O_{\P(\sE')}(2)} D$
is non-special.
Therefore, it follows from the restriction exact sequence
\[
  0 \to
  \O_{\P(\sE')}(1)
  \to \O_{\P(\sE')}(2)
  \to \restr {\O_{\P(\sE')}(2)} D
  \to 0
\]
that
\[
  h^1(S'', \O_{\P(\sE')}(1))
  \geq
  h^1(S'', \O_{\P(\sE')}(2)).
\]
We thus conclude by \cite[Lemma~1]{CaCiFlMi} that
 $\sE' = \sL^{-1} \oplus \O_\Gamma$,
and $S''$ is mapped by the linear system $|\O_{\P(\sE')}(1)|$ to the
cone over $\Gamma$ in its embedding defined by $|\sL^{-1}|$.
The conclusion follows, since $\bar C$ is a member of
$|\O_{\P(\sE')}(2)|$.

It remains to explain how to adapt these arguments when $e< 0$.
In this case $\sA\otimes\sL$ and $\sA\otimes\sL^{\otimes 2}$ are non-special,
and then the 
linear normality of $C$ implies in the same way as above that
$h^1(\Gamma,\sA)=q$, hence $\sA = \O_\Gamma$ and $\fa=0$.
Thus, $\bar C$ is a member of $|\O_{\P(\sE)}(2)|$.
One has
\[
  \deg(\sE) = -e = \fa-e \geq 2q+2,
\]
and
$h^1(\O_{\P(\sE)}(1)) \geq q$ by the exact same argument as above.
It therefore follows yet again from \cite[Lemma~3.5]{CaCiFlMi}
that $\sE$ is split, which  contradicts the assumptions that
$\sE$ is normalized and $e<0$.
\end{proof}

Conversely, virtually
every curve which is a double cover has a linearly normal projective
model which is a hyperplane section of a
surface ruled in conics, as the following example shows.

\begin{expl}
Let $\Gamma$ be a smooth curve of genus $q$,
and $\pi: C \to \Gamma$ a smooth double cover of genus $g$, with
branch divisor $B \subset \Gamma$. Correspondingly, there exists a line
bundle $\sG$ on $\Gamma$ such that
$\sG^{\otimes 2} = \O_\Gamma(B)$.
Then $C$ is a divisor in the surface $S = \P(\O_\Gamma\oplus \sG)$, member
of the linear system $|\O_{\P(\O_\Gamma\oplus \sG)}(2)|$, with normal bundle
$\restr {\O_S(C)} C = \pi^*(\sG)^{\otimes 2} = \pi^* (\O_\Gamma(B))$.
If $\deg (\sG) > 2q-2$,
the linear system $|\O_{\P(\O_\Gamma\oplus \sG)}(1)|$ maps $S$
to a cone over $\Gamma$, and $C$ in its embedding defined by
$|\pi^*( \O_\Gamma(B))|$ is linearly normal, and a hyperplane section
of the $2$-Veronese re-embedding of this cone;
this is a particular case of Example \ref{ex:mu-cones} below.
\end{expl}

\noindent
Using Theorem~\ref{thm:segre} and
Lemma~\ref{prop:irr} together,
one obtains the following.

\begin{cor}\label{cor:irr}
Let $S$ be an irreducible and non-degenerate, linearly normal, irregular surface
of degree $d$ in $\P^n$. Assume that the hyperplane sections of $S$
are smooth, linearly normal, of genus $g \geq 7$.
If $d>3g-3$, then $S$ is either the 2-Veronese re-embedding of a cone or a simple internal projection thereof.
\end{cor}

\noindent
(Note that $3g-2\geq 2g+5$ if and only if $g\geq 7$).

If we drop the assumption that
$d>3g-3$, then there are examples with
$\mu>2$ of irregular surfaces with linearly normal hyperplane
sections
(compare with Lemma \ref{prop:irr}),
where $\mu$ is as defined at the beginning of
Section~\ref{S:irreg}.
Moreover, some of these examples have empty bi-adjoint
system (compare with Section~\ref{s:empty-2adj}).
We check below that certain $3$-Veronese re-embeddings of cones 
provide such examples.

\begin{expl}
\label{ex:mu-cones}
Let us consider $\mu$-Veronese re-embeddings of cones, for arbitrary
$\mu >1$.
Let $\Gamma$ be a smooth curve of genus  $q>0$, and let $\sG$ be a line bundle
on $\Gamma$ of degree  $e>2q-2$. 
Let $S = {\P(\O_\Gamma\oplus \sG)}$,
and consider a smooth curve $C\subset S$
defined by a section of $\O_{\P(\O_\Gamma\oplus \sG)}(\mu)$.
Beware that in general a $\mu$-tuple cover of $\Gamma$ may not be
realized in this way.

We shall see the following:\\
(i) the projective curve image of $C$ by the linear
system $|\O_{\P(\O_\Gamma\oplus \sG)}(\mu)|$ is linearly normal;
in other words, 
the hyperplane sections of the $\mu$-Veronese re-embedding of
the cone over $(\Gamma,\sG)$
(i.e., the projective curve image of $\Gamma$ by $|\sG|$)
are linearly normal;
\\(ii) for all $m$ such that $2m>\mu$, the $m$-adjoint linear system
$|mK_{S'}+C|$ is empty;
\\(iii) one has $d>2g-2$, where $d = C^2$ and $g$ is the genus of $C$.

Let us first compute the genus $g$ of $C$.
For numerical equivalence, one has
\[
  C \equiv \mu D
  \quad \text{and}
  \quad
  K_S \equiv -2D+(2q-2+e)F,
\]
where $D$ is the divisor class of $\O_{\P(\O_\Gamma\oplus \sG)}(1)$
(hence $D^2=e$),
and $F$ is the numerical class of the fibres of $S\to \Gamma$.
Thus,
\begin{equation}
\label{eq:g-mu-cones}
  \begin{aligned}
    2g-2
    &= (K_S+C)\cdot C
    \\
    &= \mu D\cdot \bigl[
      (\mu-2)D+(2q-2+e)F
      \bigr]
    \\
    &= \mu\,
      \bigl[
      2(q-1)+(\mu-1)e
      \bigr].
  \end{aligned}
\end{equation}
On the other hand the degree $d$ of $C$ in the embedding defined by
$|\O_S(C)|$ is $C^2=e\mu^2$, and thus
\begin{align*}
  d > 2g-2
  &\iff
    \mu e > 2(q-1)+(\mu-1)e
  \\
  &\iff
    e > 2q-2,
\end{align*}
which proves (iii).
As a remark, note that
\begin{align*}
  d > 3g-3
  &\iff
    2e\mu > 3\,\bigl[
    2(q-1)+(\mu-1)e
    \bigr]
  \\
  &\iff
    (\mu-3)e < -6(q-1),
\end{align*}
so, since $q>0$, one has $\mu \leq 2$ if $d>3g-3$, as predicted by
Lemma \ref{prop:irr}.

One has
\[
  (C+mK_S)\cdot F = \mu-2m,
\]
which proves (ii).

Finally, let us prove (i).
Since  $e>2q-2$,  all positive multiples of $\sG$
are non-special, hence
\begin{align*}
  h^0(S,\O_{\P(\O_\Gamma\oplus \sG)}(\mu))
  &= \sum_{0\leq k \leq \mu} h^0(\Gamma,\sG^{\otimes k})
  \\
  &= 1 + \frac 1 2 \mu(\mu+1)e + \mu (1-q).
\end{align*}
On the other hand,  if $\pi: C \to \Gamma$ is the natural morphism,  $\pi^*(\sG)^{\otimes \mu}$ is non-special as well, because
$$
\deg (\pi^*(\sG)^{\otimes \mu})=e\mu^2=d>2g-2.
$$
Therefore
\begin{equation*}
  h^0(C,\pi^*(\sG)^{\otimes \mu}) = \mu^2e-g+1,
\end{equation*}
and
\begin{align*}
  h^0(C,\pi^*(\sG)^{\otimes \mu}) = h^0(S,\O_{\P(\O_\Gamma\oplus \sG)}(\mu))-1
  &\iff
    \mu(\mu+1)e + 2 \mu (1-q)
    = 2\mu^2e-(2g-2)
  \\
  &\iff
    2g-2 = \mu (\mu-1)e+2\mu(q-1).
\end{align*}
Since the last equality indeed holds, see \eqref{eq:g-mu-cones},
$C$ is linearly normal as we wanted.
\end{expl}

\section 
{The rational case: first results}\label{sec:rat}

\noindent
In this section we consider the case when $S$ is rational and
$d \geq n$.

\begin{prop}\label{basepoints2}
Let $S, S',H$ be as  in Section \ref {ssec:setup},
and assume that
$S$ is rational,
and $d\geq n$ (i.e., $a\geq 0$). 
One has:  \\
\begin{inparaenum}[(i)]
\item the hyperplane sections of $S$ are linearly normal,
   $g=a+1=h^0(K_{S'}+H)$, and $S$ is not a scroll;\\
\item  if  $a=0$ (equivalently $g=1$), then  $K_{S'}+H \sim 0$;\\
\item  if  $a=1$  (equivalently $g=2$),  then
  $|K_{S'}+H|$ is a base-point-free pencil of rational curves,
  hence  $(K_{S'}+H)^2=0$ and $h^0(S', 2K_{S'}+H)=0$;\\
\item   if $|K_{S'}+H|$ is composed with a pencil  $|G|$, then
  $G^2=0$ and $K_{S'}+H\sim (g-1)G$, hence $(K_{S'}+H)^2=0$,
  $K_{S'}\cdot G=-2$, and $H\cdot G=2$; thus $S$ is ruled
  by conics, and the curves in $|H|$ are hyperelliptic;
  moreover,  $h^0(S', 2K_{S'}+H)=0$;\\
\item if $(K_{S'}+H)^2>0$, then $|K_{S'}+H|$ is not composed with a
  pencil,
  $a\geq 2$ (equivalently $g\geq 3$),
  and the morphism
  $\varphi_{|K_{S'}+H|}$ maps $S'$ onto a  non-degenerate surface
  in $\bP^{g-1}$,  and   $K_{S'}^2\geq n-2g+1$  (or equivalently
  $(K_{S'}+H)^2\geq g-2=a-1$); moreover,
 if  equality holds, then the morphism $\varphi_{|K_{S'}+H|}$ is
 birational onto its image, which is a surface $S''$ of minimal degree
 $g-2=a-1$   in $\bP^{g-1}$, and $h^0(S', 2K_{S'}+H)=0$.
 In addition, in this case, if $S''$ is neither $\P^2$ nor the
 Veronese surface $V_2 \subset \P^5$, then $S$ has a $1$-dimensional
 family of rational cubic curves.
 \end{inparaenum}
\end{prop}

\begin{proof}
(i) follows from Proposition~\ref{qd}, \eqref{qd-lem:adj},
Proposition \ref {qd}, \eqref{qd-i} and \eqref{qd-v},
and  Remark~\ref{rem:obs2}.

(ii) follows  from the fact that $h^0(K_{S'}+H)=g=1$, together with
Proposition \ref {p:bpf-rat}.

If $a=1$, then $h^0(K_{S'}+H)=g=2$ hence
$|K_{S'}+H|$ is a pencil. It is base-point-free by Proposition \ref
{p:bpf-rat}, so $(K_{S'}+H)^2 = 0$.
Similarly,
if $|K_{S'}+H|$ is composed with a pencil $|G|$, since
it is base-point-free,  we have $G^2=0$;
moreover, $K_{S'}+H\sim (g-1)G$, because
$h^0(K_{S'}+H)=g$.
In all cases (setting $G=K_{S'}+H$ if $g=2$),
$0=G\cdot (K_{S'}+H)=G\cdot K_{S'}+G\cdot
H>G\cdot K_{S'}$. Then $G\cdot K_{S'}=-2$, which implies that the
curves in $|G|$ are rational, and $H\cdot G=2$.
Finally, $(2K_{S'}+H)\cdot G = -2$, so $2K_{S'}+H$ is not effective.
This proves both (iii) and (iv).

(v) The image of $\varphi_{|K_{S'}+H|}$ is a non--degenerate surface
in $\bP^{g-1}$.
By Proposition~\ref{lem:due}, \eqref{due-vi}, 
$(K_{S'}+H)^2 \geq g-2$. If equality holds, then
$\varphi_{|K_{S'}+H|}$ is birational onto its image, which is a surface
$S''$ of minimal degree $g-2$,
$p_a(K_{S'}+H) = 0$ and,
by \eqref{due-iv} of Proposition~\ref{lem:due},
$h^0(S',2K_{S'}+H)=0$.
If $S''$ is neither $\P^2$ nor the
Veronese surface $V_2 \subset \P^5$, then it is a scroll, and the
family of rational cubics on $S$ corresponds to the ruling of $S''$.
\end{proof}

\begin{prop}\label{conics}
\label{image}
Let $S, S',H$ be as  in Section \ref {ssec:setup}, and assume that $S$
is rational, $g\geq 1$,
and  $(K_{S'}+H)^2>0$.
Then:\\
\begin{inparaenum}[(i)]
\item\label{conics-0}  $S$ is not ruled by conics;\\
\item\label{conics-i}
$0\leq p_a(K_{S'}+H)= K_{S'}^2-n+2a+1$, and $h^0(2K_{S'}+H)= p_a(K_{S'}+H) =  K_{S'}^2-n+2a+1$;\\
\item\label{conics-ii}
  If $(K_{S'}+H)^2>a-1$ (i.e., $K_{S'}^2\geq n-2a$), then $S$ has no $1$-dimensional family of  rational curves of degree $\delta<4$.\\
\end{inparaenum}
\end{prop}

\begin{proof}
The assumption $g\geq 1$ is equivalent to $h^0(S', K_{S'}+H)>0$
(see Proposition~\ref{qd}, \eqref{qd-lem:adj}). 

Assertion \eqref{conics-0} follows  from  Lemma~\ref{grau}.
  
Assertion \eqref{conics-i} follows from Proposition~\ref{lem:due},
\eqref{due-iii} and \eqref{due-v}.

For assertion \eqref{conics-ii}: If $(K_{S'}+H)^2>a-1=g-2$ then,
by Proposition~\ref{lem:due}, \eqref{due-vi},
$p_a(K_{S'}+H)=h^0(2K_{S'}+H)>0$.
Suppose  that $S$ 
has a $1$-dimensional family of  rational curves of degree
$\delta<4$. Let  $L$  be the pull back to $S'$ of a general member of
the family. Then $H\cdot L\leq 3$, and $L^2\geq 0$. Since $L$ is a
rational curve, $K_{S'} \cdot L\leq -2$ by adjunction.  Hence
$(2K_{S'}+H) \cdot L< 0$, which is impossible because
$h^0(2K_{S'}+H)>0$ and $L$ is nef.  
\end{proof}

\section{The rational case: empty biadjoint system}
\label{sec:empt}

In this section we go on considering the same situation as in Section
\ref {sec:rat}, with some additional assumptions.

\begin{thm}\label{thm:class1}
Let $S, S',H$ be as  in Section \ref {ssec:setup},
and assume that $S$ is rational and $d\geq n$, i.e., $a\geq 0$.  
If $h^0(S', 2K_{S'}+H)=0$,
then one of the following cases occurs:\\
\begin{inparaenum}[(a)]
\item\label{class1-DP} $g=1$ and $S$ is a (weak) Del Pezzo surface;\\
\item\label{class1-V4} $g=3$ and $S$ is the Veronese surface $V_4$, or a simple internal projection thereof;\\
\item\label{class1-V5} $g=6$ and $S$ is the Veronese surface $V_5$, or a simple internal projection thereof;\\
\item\label{class1-hell} $S$ is a surface with hyperelliptic sections,
  which is
  either the degree $4g+4$ surface image of $\bF_e$ by the linear system 
   $|2E+(g+1+e)F|$,  where $0\leq e \leq g+1$, 
  or a simple internal projection thereof;\\
\item\label{class1-3gon} $S$ is a surface with trigonal sections,
  which is either the degree $3g+6$ surface image 
  of $\bF_e$ by a linear system of the form $|3E+(h+e+2)F|$,
  where $e\geq 0$ and $h\geq \max \{2e-2,e\}$ are integers such that
  $g=2h-e+2$, or a simple internal projection thereof.
\end{inparaenum}

Conversely, in all these cases,  $h^0(S', 2K_{S'}+H)=0$.
\end{thm} 

\begin{proof} By Proposition \ref {p:bpf-rat}, $|K_{S'}+H|$ is
base-point-free of dimension $g-1\geq 0$.
By Proposition \ref{basepoints2}, (ii) and (iv),
if $g=1$ then we are in
case \eqref{class1-DP}, and if $|K_{S'}+H|$ is composed with a pencil
then we are in case \eqref{class1-hell}.  Else, $(K_{S'}+H)^2>0$, and
the general curve in $|K_{S'}+H|$ is smooth and irreducible. Since
$|2K_{S'}+H|$ is empty, the curves in $|K_{S'}+H|$ have genus 0, hence
$(K_{S'}+H)^2=\dim (|K_{S'}+H|)-1=g-2=a-1$. Consider the morphism
$$\varphi_{|K_{S'}+H|}: S'\longrightarrow S''\subset \bP^{g-1}$$
which, by Proposition \ref{basepoints2}, (v),
is birational onto its image,
which is a surface of minimal degree $g-2$ in $\bP^{g-1}$.  Hence $S''$
is either the plane (then $g=3$), or the Veronese surface $V_2$ in
$\bP^5$ (then $g=6$), or a rational normal scroll. In any event, we
denote by $\phi: X\longrightarrow S''$ the minimal
desingularization of $S''$ and we set $\sL=\phi^*(|\mathcal
O_{S''}(1)|)$. Then we have a birational map $\varphi: S'\dasharrow
X$, which is a morphism if $S''$ is not a cone, and $|K_{S'}+H|$ is
the pull--back via $\varphi$ of the linear system $\sL$. We have the
following cases:\\
\begin{inparaenum}
\item [(i)] $X=\bP^2$ and $\sL=|\sO_{\bP^2}(1)|$ (then $g=3$);\\
\item [(ii)] $X=\bP^2$ and $\sL=|\sO_{\bP^2}(2)|$ (then $g=6$);\\
\item [(iii)] $X=\bF_e$ and $\sL=|\sO_{\bF_e}(E+hF)|$, with $h\geq e$,
but not $e=h=1$.
\end{inparaenum}
Note that $\phi: X\longrightarrow S''$ is an isomorphism and $\varphi=\varphi_{|K_{S'}+H|}$ unless we are in case (iii) and $h=e$ (i.e., if $S'$ is a cone). 

In case (i) the general curve $C\in |H|$, that has genus 3,  is mapped
via $\varphi$ to a smooth plane curve of degree $2g-2=4$. Hence the
surface $S\subset \bP^{n}$ is the image of $\bP^2$ via a linear system
of generically smooth plane quartics, possibly with simple base points,
and the system has to be complete under the condition of containing
these base points.  Thus we are in case \eqref{class1-V4}.  In case
(ii), by the same argument  we end up in case \eqref{class1-V5}.  

Suppose  we are in case (iii). Then the surface $S''$ is a rational normal scroll of degree $2h-e$  in $\bP^{2h-e+1}$, hence $g=2h-e+2$ and the general curve $C\in |H|$ is mapped via $\varphi_{|K_{S'}+H|}$ to a canonical curve of degree $2g-2=4h-2e+2$. We abuse notation and denote by $C$ the image of $C$ on $X$. One has $C\cdot F=3$, thus $C\sim 3E+kF$, with $k$ a suitable integer. Since 
$$
4h-2e+2=2g-2=C\cdot (E+hF)=(3E+kF)\cdot (E+hF)=k+3h-3e,
$$
one has $k=h+e+2$, hence $C\sim 3E+(h+e+2)F$. Since $C$ is nef, one has $C\cdot E\geq 0$, which gives $h\geq 2e-2$. Moreover   
$$
C^2=(3E+(h+e+2)F)^2=6h-3e+12.
$$
Note that we can be in the cone case only if $e=h=2$ (hence $g=4$).
In any event,  we are in case \eqref{class1-3gon}.

Lastly, the fact that $|2K_{S'}+H|$ is empty in all these cases
follows by a direct examination.
\end{proof}

\begin{rem}\label{rem:bp}
In case \eqref{class1-V4} of Theorem \ref {thm:class1}, let $b$ be the length of the 0--dimensional curvilinear scheme $Z$ from which we make the internal projection of $V_4\subset {\bf P}^{14}$. Since $14-b= n\geq 3$, one must have $b\leq 11$. On the other hand, any curvilinear scheme $Z$ of length $b\leq 11$ lying on a smooth irreducible plane quartic gives independent conditions to plane curves of degree 4, so every such scheme is allowed.

In case \eqref{class1-V5}, let again $b$ be the length of the
0--dimensional curvilinear scheme $Z$ from which we make the internal
projection of $V_5\subset {\bf P}^{20}$. Since $20-b= n\geq 3$, one
must have $b\leq 17$. Any curvilinear scheme $Z$ of length $b\leq 14$
lying on a smooth irreducible plane quintic $D$ gives independent
conditions to plane curves of degree 5, so every such scheme is
allowed. If $15\leq b\leq 17$, $Z$ is allowed if and only if
there exists a non-special divisor $Z'$ on $D$ such that 
$Z+Z'\in |\mathcal O_D(5)|$,
i.e., $Z'$ is not contained in a conic. In this case $Z$ still gives  
independent conditions to plane curves of degree 5. 
\end{rem}

\section{The rational case: empty triadjoint system}
\label{sec:nonempt}

In this section we again consider the situation of
Section~\ref{sec:rat} with some additional assumptions, 
but the latter are different from the additional assumptions we
considered in Section~\ref{sec:empt}.

\begin{setup}
\label{setup-10}
  Let $S, S',H$ be as  in Section \ref {ssec:setup},
  and assume that $S$ is rational,  and $d\geq n$, i.e., $a\geq 0$.  
This time we suppose that the bi-adjoint system $|2K_{S'}+H|$ is
non-empty, whereas the tri-adjoint system $|3K_{S'}+H|$ is empty. This
latter condition is verified if $d>3g-3$ (see Lemma \ref
{lem:ad2_intro}).

As usual, $|K_{S'}+H|$ is base-point-free by Proposition \ref
{p:bpf-rat}; moreover, by Proposition \ref{basepoints2}, since
$|2K_{S'}+H|$ is non-empty, 
$|K_{S'}+H|$ is not composed with a pencil and
$(K_{S'}+H)^2>0$. In particular, $K_{S'}+H$ is big and nef, and the general curve $M\in |K_{S'}+H|$ is smooth and irreducible. 
We set
$$
|2K_{S'}+H|=\Phi_2+|M_2|
$$
where $|M_2|$ is the (possibly empty) movable part, and $\Phi_2$ is
the fixed part. One has 
$M\cdot \Phi_2=0$.
\end{setup}

\begin{lem}
\label{l:marga-proj}
  There exists a $(-1)$-curve  $E$ on $S'$ such that $H\cdot E=1$ if and
  only if $S$ is   a simple    internal projection.
  
   Furthermore, if there is a $(-1)$-divisor $A$  such that $H\cdot A=1$, then $S$ is a simple  internal projection. 
  
\end{lem}

\begin{proof}

 If $S\subset \bP^n$ is  a simple internal projection, there is a surface  $\Sigma\subset \bP^{n+1}$ such that $S$ is obtained  by  projecting $\Sigma$ from a smooth point $p\in \Sigma$ in such a way that the projection $\Sigma\dasharrow S$ is birational.  Let $\Sigma'\longrightarrow \Sigma$ be the blow--up of $\Sigma$ at $p$, with exceptional divisor $E$ that is a $(-1)$--curve. The projection $\Sigma\dasharrow S$ determines an isomorphism $\sigma: \Sigma'\longrightarrow S$, and, abusing notation, we denote by $E$ the image of $E$ via $\sigma$, that is a line and it is contained in the smooth locus of $S$. Then the pull back of $E$ via $\pi: S'\longrightarrow S$, that we still denote by $E$ by abuse of notation, is a $(-1)$--curve and $H\cdot E=1$.  Hence the ``if'' part has been proved, so let us prove the other implication.

Let
$E$ be a $(-1)$-curve as in the statement, and set $\tilde H =
H+E$.

We claim that $h^0(S',\tilde H) = h^0(S',H)+1$:
this follows from the exact sequence
\begin{equation*}
  0 \to \O_{S'}(H) \to \O_{S'}(\tilde H) \to \O_E(\tilde H) \simeq
  \O_{E}\to 0,
\end{equation*}
and the fact that $h^1(S',H)=0$, given by
Proposition \ref {qd}, \eqref {qd-i}. Now,
consider the image $\Sigma$ of $S'$ via the morphism $\phi$ determined by the (base point free) linear system $|\tilde H|$. Then $\phi$ contracts $E$ to a smooth point $p$ of $\Sigma$, and $S$ is the projection of $\Sigma$ from $p$. Accordingly, $S$ is a simple internal projection.

  To prove the last assertion, it suffices to show that one component of $A$ is a $(-1)$-curve.  
Since  $|K_{S'}+H|$ is nef and big and  $(K_{S'}+H)\cdot  A=0$, the intersection form on the components of $A$ is negative definite, hence by the adjunction formula $K_{S'} \cdot \theta\geq -1$ for every component $\theta$ of $A$.  From $K_{S'}\cdot A=-1$, we conclude that $A$ contains a  curve  $\theta_0$ such that  $K_{S'}\cdot \theta_0=-1$.  Since $\theta_0^2<0$, $\theta_0$ is a $(-1)$-curve.  \end{proof}

\noindent
We will use the following:

\begin{lem}\label{lem:caci}
Let $T,X$ be smooth, irreducible,
projective surfaces, and let $\varphi:  T  \dasharrow X$ be a
birational map.  
Let us consider the minimal resolution of the indeterminacies of
$\varphi$ by the following diagram 
$$
\xymatrix{
& Z
\ar[dl]_f \ar[dr]^{f'}
\\
T \ar@{-->}[rr]_\varphi && X.
}
$$
Let $D$ be a smooth curve on $T$. Then one has
\begin{equation}\label{eq:wol}
\varphi_*(D+K_{T})= \varphi_*(D)+K_X+A,
\end{equation}
where $A$ is the image via $f'_*$ of some $f$-exceptional divisor.
If  $|D|$  is base-point-free, one has  $\varphi_*( D  )\cdot A=0$. 
\end{lem}

\begin{proof} This follows from \cite[Formula (13)] {Caci}.\end{proof}

\begin{thm}\label{thm:class2}
Consider Setup~\ref{setup-10}, and assume that $|M_2|$ is empty.
Then:\\
\begin{inparaenum}[(a)]
\item\label{class2-a}
  if $g\geq 4$ , then $S$ is the image by the Veronese map $v_2$ of a
 (weak) Del Pezzo surface of degree $g-1$ in $\bP^{g-1}$, or a simple internal projection thereof, as in
\cite[Ex. 3.3(a)]{CD}.
In this case one has $g\leq 10$;\\
\item\label{class2-b}
 if $g=3$, then $S$ can be obtained as the image of the plane via
  the map determined by a linear system of the form $(6; 2^7)$,
  possibly with some further simple base points.
\end{inparaenum}
\end{thm}

\begin{proof}
We assume without loss of generality that $S$ is not an internal
projection. By Lemma~\ref{l:marga-proj}, this tells us that for every
$(-1)$-curve $E$ on $S'$, $H\cdot E>1$, hence $M\cdot E >0$.
Since $\dim (|2K_{S'}+H|)=0$,
the curves in $|M|$ have arithmetic genus $1$,
and $\dim(|M|)=g-1=M^2$.

\eqref{class2-a}
Assume that $g\geq 4$. Then the map defined by $|M|$ is birational.
Since moreover there are no $(-1)$-curves $E$ such that $M\cdot
E=0$,
we can apply Proposition \ref {p:bpf-rat} to $|M|$, and conclude that
$|K_{S'}+M|=|2K_{S'}+H|$ has no base points, so that $\Phi_2=0$,
and therefore $2K_{S'}+H \sim 0$, i.e., $H\sim -2K_{S'}$.
Then
$2g-2=2K^2_{S'}$, so that $g=K^2_{S'}+1\leq 10$ and, since $g\geq 4$,
one has $K^2_{S'}\geq 3$. Moreover, $h^1(S', -K_{S'})=h^1(S',
K_{S'}+H)=0$, and therefore, by the Riemann--Roch theorem, $\dim
(|-K_{S'}|)=K^2_{S'}\geq 3$. Then $|-K_{S'}|$ is a linear system of
curves of genus 1, and $\varphi_{|-K_{S'}|}$ maps $S'$ to a (weak) Del
Pezzo surface of degree $K^2_{S'}$. The assertion follows.

\eqref{class2-b}
Assume that $g=3$. In this case, $|M|$ is a net of elliptic curves. We cannot
apply Proposition~\ref{p:bpf-rat}, but we will verify by hand that an
analogous conclusion holds.
By the
classification theorem of linear systems of elliptic curves on a
rational surface, see \cite [p. 97]{Ca}, there is a birational map
$\varphi: S'\dasharrow X$, where $X$ is the
blow--up of $\P^2$ at 7 suitable (proper or infinitely near) points
$p_1,\ldots, p_7$ forming a curvilinear scheme, and $|M|$ is the pull back via
$\varphi$ of the anticanonical system of $X$.

Let $\psi: X\longrightarrow \bP^2$ be the blow-up map, and
consider the composite birational map $\gamma=\psi\circ \varphi:
S'\dasharrow \bP^2$.  
Then $\gamma$ maps
$|M|$ to the linear system $\sM=(3;1^7)$ of curves of degree 3 passing
through the points $p_1,\ldots, p_7$. 
Set $\Psi=\psi_*(A)$, in the notation of Lemma \ref {lem:caci}
applied to the map $\varphi: S'\dasharrow X$ and the divisor $H$ on $S'$.
The map $\gamma$ may contract some $(-1)$--curves to (proper or
infinitely near) further points $p_{8},\ldots, p_k$, different from
$p_1,\ldots, p_7$, which are not base points for $\mathcal M$. The
divisor $\Psi$ belongs to some 0--dimensional linear system
$(\ell;m_1,\ldots, m_k)$. 
The linear system $|H|$ is mapped by $\gamma$ to a linear system $\sH$ of the form 
$(\ell';m'_1,\ldots, m'_k)$.  We may consider all these linear systems on the blow--up of $\bP^2$ at the points $p_1,\ldots, p_k$. 

The system $\mathcal M-\Psi$, that is $(3-\ell;1-m_1,\ldots, 1-m_7, -m_{8},\ldots ,-m_k)$, is the adjoint system to $\mathcal H$ by \eqref {eq:wol}; this adjoint system has no base points by Proposition \ref {p:bpf-rat}, and has dimension 2. This implies that $m_8=\cdots =m_k=0$, otherwise $\mathcal M-\Psi$ would have in its base locus the exceptional curves corresponding to the points $p_8,\ldots, p_k$.

This implies that 
$$
\ell'=6-\ell, \,\,\, m'_i=2-m_i,\,\,\, \text{for}\,\,\, 1\leq i\leq 7, \,\,\,  m'_i=1,\,\,\, \text{for}\,\,\, 8\leq i\leq k.
$$
We have the following two important pieces of information: (1) the
curves in $\mathcal M$ cut out on the general curve in $\mathcal H$,
off the base points, the complete canonical series;  (2) the curves in
$\mathcal H$ intersect $\Psi$ only at the base points (by  Lemma \ref
{lem:caci}). Property (1) implies 
\begin{equation}\label{eq:ct}
  3(6-\ell)-\sum_{i=1}^7(2-m_i)=4,
  \quad \text{hence}\quad 3\ell- \sum_{i=1}^7m_i=0,
\end{equation}
which means that the curves in $\mathcal M$ intersect $\Psi$ only at the base points. Property (2) reads
$$
\ell(6-\ell)-\sum_{i=1}^7 m_i(2-m_i)=0
$$
which, taking into account \eqref {eq:ct}, implies
\begin{equation}\label{eq:cy}
\ell^2-\sum_{i=1}^k m_i^2= 0.
\end{equation}
But, since the curves in $\mathcal H$ intersect $\Psi$ only at the
base points,  $\Psi\neq 0$ implies that
$$
\ell^2-\sum_{i=1}^k m_i^2<0;
$$
thus, \eqref {eq:cy} yields that $\Psi=0$, which implies that
$\varphi$ is regular.
In the upshot, $\ell=0$ and $m_i=0$ for all $i=1,\ldots,k$,
therefore $\mathcal H$ is the system $(6; 2^7, 1^{k-7})$, which proves
the assertion.
\end{proof}

\begin{rem}\label{rem:dp} Theorem \ref {thm:class2}, (i), tells us
  that the surface $S$ can be obtained as the image of the plane via
  the map determined by a linear system of the form $(6; 2^h)$, with
  $h\leq 6$, possibly with some further simple base points, or as the
  2--Veronese image of a quadric in $\bP^3$ or a simple internal
  projection thereof. This shows the similarity of part (i) of Theorem \ref {thm:class2} with part (ii). 
\end{rem}

\noindent
Next we consider the case in which $|M_2|$ is non-empty.

\begin{thm}\label{thm:class3}
Consider Setup~\ref{setup-10}.
If $|M_2|$ is non-empty, then one of the following cases occurs: \\
\begin{inparaenum}[(a)]
\item\label{class3:V7}
$S$ is the surface image of
$\bP^2$ by a linear subsystem of
$|\sO_{\bP^2}(7)|$ determined by $s$ double base points,
or an internal projection thereof from a scheme $Z$ of length $t$,
where $s$ and $t$ are non-negative integers such that
$$d=49-4s-t \quad \text{and}\quad g=15-s; $$
\item\label{class3:V8}
$S$ is the surface image of
$\bP^2$ by a linear subsystem of
$|\sO_{\bP^2}(8)|$ determined by $s$ double base points,
or an internal projection thereof from a scheme $Z$ of length $t$,
where $s$ and $t$ are non-negative integers such that
\[
  d=64-4s-t\quad \text{and} \quad g=21-s;
\]
\item\label{class3:4gon}
$S$ is the surface image of $\bF_e$ 
by a linear subsystem of $|4E+(h+2e+4)F|$ determined by $s$ double base
points,
or an internal projection thereof from a scheme $Z$ of length $t$,
with integers $e\geq 0$, $h\geq \max \{2e-4,e-2\}$, $s\geq 0$,
and $t\geq 0$ such that
\[
  d=8h+32-4s-t \quad \text{and} \quad g=3h+9-s;
\]
\item\label{class3:5gon}
$S$ is the surface image of $\bF_e$ 
by a linear subsystem of $|5E+(h+2e+4)F|$ determined by $s$
double base points, or an internal projection thereof from a scheme $Z$
of length $t$,
with integers $e\geq 0$, $h\geq \max \{3e-4, 2e-2, e\}$, $s\geq
0$, and $t\geq 0$ such that
$$d=10h-5e+40-4s-t\quad \text{and}\quad g=4h-2e+12-s.$$
\end{inparaenum}
\end{thm}

\begin{proof}
As before,
we assume without loss of generality that $S$ is not an internal
projection. By Lemma~\ref{l:marga-proj}, this tells us that for every
$(-1)$-curve $E$ on $S'$, $H\cdot E>1$, hence $M\cdot E >0$.
As stated in Setup~\ref{setup-10},
$|M|=|K_{S'}+H|$ is base-point-free, $M^2>0$, and
the general curve  in $|M|$ is smooth.  Since $|3K_{S'}+H|$ is
empty, any irreducible curve contained in a curve of
$|2K_{S'}+H|=\Phi_2+|M_2|$ is smooth and rational.
We start by proving various claims.

\begin{claim}
The divisor $2K_{S'}+H=K_{S'}+M$ is nef.
\end{claim}

\begin{proof}[Proof of the claim]
Suppose there exists an irreducible curve $\theta$ such that
$\theta \cdot (K_{S'}+M) <0$.
Then $\theta ^2 <0$.
Moreover, $K_{S'} \cdot \theta \leq \theta \cdot (K_{S'}+M) <0$.
This implies that $\theta$ is a $(-1)$-curve and $M\cdot \theta\leq 0$,
which is excluded by our assumption that $S$ is not an internal
projection. 
\end{proof}

\begin{claim}
One has $M^2 \geq 3$.
\end{claim}

\begin{proof}[Proof of the claim]
One has $M^2>0$, so it suffices to prove that
$M^2 \neq 1, 2$.

Let us first prove that $K_{S'}\cdot M <0$. One has
$K_{S'}\cdot M = K_{S'} \cdot (K_{S'}+H) < K_{S'}^2$.
Thus, if $K_{S'}^2 \leq 0$, the wanted inequality holds.
Else, $K_{S'}^2 >0$. Then, by Riemann--Roch, $-K_{S'}$ is effective and, by the
index theorem,
$-K_{S'}\cdot (K_{S'}+H) >0$, which is the required inequality.

Assume by contradiction that $M^2=1$ or $2$. Since $K_{S'}\cdot M <0$
and the curves in $|M|$ are not rational, we must have
$(K_{S'}+M)\cdot M=0$. Besides,
$(K_{S'}+M)^2 \geq 0$ because $K_{S'}+M$ is nef.
Therefore, by the index theorem,
$K_{S'}+M \lineq 0$, i.e., $2K_{S'}+H\lineq 0$. Then
$|M_2|$ is empty, contrary to the assumption.
\end{proof}

\begin{claim}
The divisor $2K_{S'}+H=K_{S'}+M$ is base-point-free.
In particular, $\Phi_2=0$.
\end{claim}

\begin{proof}[Proof of the claim]
Let us first consider the case in which
$(2K_{S'}+H)^2=0$.
Then, by Proposition~\ref{composto}, there exists a base-point-free
pencil $|G|$ of rational curves such that
$|2K_{S'}+H| = |hG|$ for some integer $h>0$, and the assertion
follows.

Else, $(2K_{S'}+H)^2 >0$.
Then we argue in a way similar to the proof of
Proposition~\ref{connected}.
Assume by contradiction that there exists a base point $x$ of
$|K_{S'}+M|$. There is a curve $M \in |M|$ passing through $x$. Then
the canonical series of $M$ has a base point, and therefore $M$ is not
$2$-connected, by \cite[Proposition (A.7)]{cfm}.
However, by Lemma~\ref{l:1-conn}, $M$ is $1$-connected because it is
big and nef.
Therefore, by Lemma~\ref{l:1-conn},
there exists a decomposition $M=A+B$, with $A$ and $B$
effective and $1$-connected, $A\cdot B = 1$, and $A^2\leq B^2$,
and one of
the following holds:\\
\makeatletter
\newcommand*{\@greek}[1]{%
  $\ifcase#1\or\alpha\or\beta\or\gamma\or\delta\or\varepsilon
    \or\zeta\or\eta\or\theta\or\iota\or\kappa\or\lambda
    \or\mu\or\nu\or\xi\or o\or\pi\or\varrho\or\sigma
    \or\tau\or\upsilon\or\phi\or\chi\or\psi\or\omega
    \else\@ctrerr\fi$
  }
\begin{inparaenum}[(\theenumi)]
\renewcommand{\theenumi}{\@greek{\arabic{enumi}}\!\!}
\item\label{vener-a} $A^2=-1$, $M\cdot A=0$, or \\
\item\label{vener-b} $A^2=0$, $M\cdot A=1$, or\\
\item\label{vener-c} $A^2=B^2=1$, $A\equiv B$, $M^2=4$ and $M\equiv 2A$.
\end{inparaenum}
\makeatother

In case \eqref{vener-a}, $A\cdot K_{S'} \leq A\cdot M =0$.
Since $A^2=-1$ and $A$ is $1$-connected, $A\cdot K_{S'} = -1$, and $A$ is a
$(-1)$-divisor. Since $A\cdot H = 1$, this contradicts our assumption
that $S$ is not an internal projection.

In case \eqref{vener-b}, $A\cdot K_{S'} \leq A\cdot M =1$.
Thus, $A\cdot K_{S'}$ equals either $-2$, and then $p_a(A)=0$,
or $0$, and then $p_a(A)=1$.
The first case cannot happen: indeed, in that case, $A$ would move in
a pencil of rational curves intersecting $M$ in one point, and the
curves in $|M|$ would be rational, in contradiction with our
assumptions.

Therefore, we must have $A\cdot K_{S'} =0$. Since $A\cdot M = 1$, and
$|M|$ is base-point-free, $h^0(A, \restr M A) \geq 2$.
Then, by \cite[Prop.\ A.5, (ii)]{cfm},
$A$ is not 2-connected.
So, by Lemma~\ref{reider-1con},
$A$ has a decomposition $A=A_1+A_2$ in which $A_1$ is a $(-1)$-divisor
such that $A_1\cdot H = 1$, in contradiction with our assumption that
$S$ is not an internal projection.

In case \eqref{vener-c}, $K_{S'}\cdot A = \frac 1 2 K_{S'}\cdot M <0$. Since
$K_{S'}\cdot A$ is odd, it equals either $-1$, and then $p_a(A)=1$,
or $-3$, and then $p_a(A)=0$.
The latter case cannot happen: if $A$ is rational, since $A^2=1$ and
$M \lineq 2A$, $M$ must be rational as well, which is excluded.
Thus, we must have $K_{S'}\cdot A = -1$.
Then, by the index theorem, $(K_{S'}+M)^2 \leq 1$, and in fact
equality holds because we are assuming $(K_{S'}+M)^2 >0$, so
$K_{S'}+M \lineq A$, hence
$2(K_{S'}+M) \lineq M$. Finally
$H \lineq -3K_{S'}$ and therefore $|H+3K_{S'}|$ is non-empty, in
contradiction with our assumptions.
\end{proof}

We may now proceed with the proof of the theorem.
Suppose first that the curves in $|M_2|$ are reducible, hence there is
a base-point-free pencil $|G|$ of rational curves, such that
$|M_2|=|hG|$ for some positive integer $h$.  
We have a birational morphism $\varphi: S'\longrightarrow \bF_e$, for
some non--negative integer $e$, such that $|G|$ is the pull back via
$\varphi$ of the ruling $|F|$.  Since $M\sim M_2-K_{S'}$, we have
$M\cdot G=2$.  Then the image $\mathcal M$ of $|M|$ via $\varphi$ is a
fixed-component-free linear subsystem of a system of the form
$|2E+kF|$, and $\mathcal M$ may have some base points that are simple
or double. However, up to performing elementary transformations, we
can get rid of the double base points, and thus we assume without loss
of generality that
$\mathcal M$ has only $s\geq 0$ simple base points.

Note that, since $\mathcal M$ is fixed-component-free, one has   $0\leq E\cdot (2E+kF)=k-2e$, i.e., $k\geq 2e$. 
Since $K_{\bF_e}\sim -2E-(e+2)F$, we have $\mathcal M+K_{\bF_e}=(k-e-2)F$, and therefore we must have $k-e-2=h$, hence $k=h+e+2$,  and so $h+e+2\geq 2e$, i.e., $h\geq e-2$. 

Therefore, $\mathcal M$ is a linear subsystem of
$|2E+(h+e+2)F|$, with $s\geq 0$ simple base points. Thus, the image $\mathcal H$ of $|H|$ via $\varphi$, is a linear subsystem of 
$$
|\mathcal M-K_{\bF_e}|=|4E+(h+2e+4)F|
$$
with $s$ double base points and maybe $t$ simple base points.   Again we must have $0\leq E\cdot (4E+(h+2e+4)F)=h-2e+4$, so $h\geq 2e-4$.  
We are then in case \eqref{class3:4gon}.
This ends the analysis in the case when $|M_2|$ is composed with a
pencil.

Assume now that $|M_2|$ is not composed with a pencil, thus the
general curve in $|M_2|$ is smooth and irreducible.  Consider the
morphism
$\varphi_{|M_2|}: S'
\longrightarrow S''\subseteq \bP^{r}$ determined by the complete
linear system $|M_2|$ of rational curves. Then, as in the proof of
Theorem \ref {thm:class1}, the surface $S''$ is a surface of minimal
degree $r-1$ in $\bP^{r}$. We  denote by $\phi: X\longrightarrow S''$
the minimal desingularization of $S''$, and set
$\sL=\phi^*(|\mathcal O_{S''}(1)|)$. Then we have a birational map
$\varphi: S'\dasharrow X$, and $|M'|$ is the pull-back via $\varphi$
of the linear system $\sL$. We have the following cases:\\  
\begin{inparaenum}
\item [(i)] $X=\bP^2$ and $\sL=|\sO_{\bP^2}(1)|$;\\
\item [(ii)] $X=\bP^2$ and $\sL=|\sO_{\bP^2}(2)|$;\\
\item [(iii)] $X=\bF_e$ and $\sL=|\sO_{\bF_e}(E+hF)|$, with $h\geq e$,
  but not $e=h=1$.\\
\end{inparaenum}
The map $\phi: X\longrightarrow S''$ is an isomorphism and
$\varphi=\varphi_{|M'|}$, unless we are in case (iii) and $h=e$ (i.e.,
if $S'$ is a cone).  

Assume we are in case (i). Then the linear system $|M|$ is mapped by
$\varphi$ to 	a linear subsystem of $|\sO_{\bP^2}(4)|$ with $s\geq
0$ simple base points and therefore $|H|$ is mapped to a linear
subsystem $\mathcal H$ of  $|\sO_{\bP^2}(7)|$ with $s\geq 0$ double base
points and $t$ simple base points, and we are in case \eqref{class3:V7}.  Case (ii)
is analogous and leads to case \eqref{class3:V8}.  

Finally, assume we are in case (iii).  One has
$\varphi_*(|M_2|)=|E+hF|$,
so $\dim (|M_2|)=\dim (|E+hF|)=2h-e+1$,
and the genus of the curves in $|M|$ is $p=2h-e+2$. 

Let $\mathcal M=\varphi_*(|M|)$. The image of the general curve of
$\mathcal M$ via $\phi$ is a canonical curve on $S''$, that is
therefore trigonal. Hence $\mathcal M$ is a linear subsystem of a
system of the form $|3E+kF|$,  that may have a certain number $s$ of
simple base points.  We must have $0\leq E\cdot (3E+kF)=k-3e$, i.e.,
$k\geq 3e$.  Note that, if $S''$ is a cone, the   
images of the general curves in $\mathcal M$  via $\phi$ do not
contain the vertex of the cone (because they are trigonal),
and therefore $E\cdot (3E+kF)=0$, i.e., $k=3e$. 
We  have
$$
4h-2e+2=2p-2=(3E+kF)\cdot (E+hF)=3h+k-3e,
$$
hence 
\begin{equation}\label{eq:lop}	
k=e+h+2. 
\end{equation}
Since $k\geq 3e$, one has $h\geq 2e-2$.  Then $\mathcal M$ is a
linear subsystem of $|3E+(e+h+2)F|$ with $s$ simple base points. So
$|H|$ is mapped by $\varphi$ to a linear subsystem of $|5E+(2e+h+4)F|$
with $s$ double base points,
and possibly $t$ simple base points, and we are in
case \eqref{class3:5gon}.  Notice that $E\cdot (5E+(2e+h+4)F)\geq 0$,
which reads $h\geq 3e-4$. Notice also that $S''$ can be a cone only if $e=2$, because in
that case $h=e$, $k=3e$, and therefore $e=2$ follows from \eqref
{eq:lop}.
 \end{proof}

\end{document}